\documentclass[11pt, a4paper]{amsart}

\usepackage{amsfonts,amsmath,amssymb,amsthm,color,enumerate,cite}
\usepackage{a4wide}

\newtheorem{lemma}{Lemma}[section]
\newtheorem{theorem}[lemma]{Theorem}

\newtheorem{proposition}[lemma]{Proposition}

\def\XXint#1#2#3{{\setbox0=\hbox{$#1{#2#3}{\int}$ }
\vcenter{\hbox{$#2#3$ }}\kern-.6\wd0}}

\makeatletter
\newcommand*{\mint}[1]{%
  \mint@l{#1}{}%
}
\newcommand*{\mint@l}[2]{%
  \@ifnextchar\limits{%
    \mint@l{#1}%
  }{%
    \@ifnextchar\nolimits{%
      \mint@l{#1}%
    }{%
      \@ifnextchar\displaylimits{%
        \mint@l{#1}%
      }{%
        \mint@s{#2}{#1}%
      }%
    }%
  }%
}
\newcommand*{\mint@s}[2]{%
  \@ifnextchar_{%
    \mint@sub{#1}{#2}%
  }{%
    \@ifnextchar^{%
      \mint@sup{#1}{#2}%
    }{%
      \mint@{#1}{#2}{}{}%
    }%
  }%
}
\def\mint@sub#1#2_#3{%
  \@ifnextchar^{%
    \mint@sub@sup{#1}{#2}{#3}%
  }{%
    \mint@{#1}{#2}{#3}{}%
  }%
}
\def\mint@sup#1#2^#3{%
  \@ifnextchar_{%
    \mint@sub@sup{#1}{#2}{#3}%
  }{%
    \mint@{#1}{#2}{}{#3}%
  }%
}
\def\mint@sub@sup#1#2#3^#4{%
  \mint@{#1}{#2}{#3}{#4}%
}
\def\mint@sup@sub#1#2#3_#4{%
  \mint@{#1}{#2}{#4}{#3}%
}
\newcommand*{\mint@}[4]{%
  \mathop{}%
  \mkern-\thinmuskip
  \mathchoice{%
    \mint@@{#1}{#2}{#3}{#4}%
        \displaystyle\textstyle\scriptstyle
  }{%
    \mint@@{#1}{#2}{#3}{#4}%
        \textstyle\scriptstyle\scriptstyle
  }{%
    \mint@@{#1}{#2}{#3}{#4}%
        \scriptstyle\scriptscriptstyle\scriptscriptstyle
  }{%
    \mint@@{#1}{#2}{#3}{#4}%
        \scriptscriptstyle\scriptscriptstyle\scriptscriptstyle
  }%
  \mkern-\thinmuskip
  \int#1%
  \ifx\\#3\\\else_{#3}\fi
  \ifx\\#4\\\else^{#4}\fi
}
\newcommand*{\mint@@}[7]{%
  \begingroup
    \sbox0{$#5\int\m@th$}%
    \sbox2{$#5\int_{}\m@th$}%
    \dimen2=\wd0 %
    \let\mint@limits=#1\relax
    \ifx\mint@limits\relax
      \sbox4{$#5\int_{\kern1sp}^{\kern1sp}\m@th$}%
      \ifdim\wd4>\wd2 %
        \let\mint@limits=\nolimits
      \else
        \let\mint@limits=\limits
      \fi
    \fi
    \ifx\mint@limits\displaylimits
      \ifx#5\displaystyle
        \let\mint@limits=\limits
      \fi
    \fi
    \ifx\mint@limits\limits
      \sbox0{$#7#3\m@th$}%
      \sbox2{$#7#4\m@th$}%
      \ifdim\wd0>\dimen2 %
        \dimen2=\wd0 %
      \fi
      \ifdim\wd2>\dimen2 %
        \dimen2=\wd2 %
      \fi
    \fi
    \rlap{%
      $#5%
        \vcenter{%
          \hbox to\dimen2{%
            \hss
            $#6{#2}\m@th$%
            \hss
          }%
        }%
      $%
    }%
  \endgroup
}

\title [] 
 {Homogenization of a mean field game system  in the small noise limit}

\author
{Annalisa Cesaroni$^{1}$}\thanks{$^{1}$ Partially supported by Indam-Gnampa and by Progetto di Ateneo ``Traffic flow on networks: analysis and control''}\address{Annalisa Cesaroni, Padova University, Dip. Statistical Sciences, via C. Battisti 241 35131 Padova, ITALY, \\ {\tt annalisa.cesaroni@unipd.it}} \author{Nicolas Dirr$^{2}$}\thanks{$^{2}$ Partially supported by the University of Padova, the Leverhulme Trust RPG-2013-261 and EPSRC EP/M028607/1} \address{Nicolas Dirr, Cardiff University, School of Mathematics, Senghennydd Road, Cardiff, UK, CF24 4AG, {\tt DirrNP@cardiff.ac.uk}}  \author{Claudio Marchi$^{1}$}
\address{Claudio Marchi, Padova University, Dip. Ingegneria dell'Informazione, via Gradenigo 6/b, 35131 Padova, ITALY,  {\tt claudio.marchi@unipd.it}} 

\begin{document}
\maketitle
\noindent{\bf Abstract.}   This paper concerns the simultaneous effect of homogenization and of the small noise limit for a $2^{\textrm {nd}}$ order mean field games (MFG) system with local coupling and quadratic Hamiltonian.  We show under some additional assumptions that the solutions of our system converge to a solution of an effective $1^{\textrm {st}}$ order system whose effective operators are defined through a cell problem which is a $2^{\textrm {nd}}$ order system of ergodic  MFG type. We provide several properties of the effective operators and we show that in general the effective system loses the MFG structure.

\section{Introduction}
We investigate evolutive mean field games (briefly, MFG) systems in the small noise limit when the Hamilton-Jacobi equation has a rapidly varying dependence on the state variable $x$, namely, 
\begin{equation}\label{systemintro}\begin{cases}
-u_t^\epsilon-\epsilon\Delta u^\epsilon +\frac{1}{2}|\nabla u^\epsilon|^2=V\left(\frac{x}{\epsilon}, m^\epsilon \right), & x\in \mathbb{R}^n, t\in (0,T)\\
m^\epsilon_t-\epsilon \Delta m^\epsilon-{\rm div}(m^\epsilon\nabla u^\epsilon)=0,& x\in \mathbb{R}^n, t\in (0,T)
\end{cases}\end{equation}\rm
with  initial and terminal condition $u^\epsilon(x,T)=u_0(x)$ and $m^\epsilon(x,0)=m_0(x)$.

MFG systems have been introduced by Lasry and Lions in \cite{ll1, ll2,ll} in the study of the overall behaviour of a large population of (rational and indistinguishable) individuals in markets, crowd motion, etc. In their approach, in a population of $N$ agents, each single agent is driven by a dynamics perturbed by a random noise and aims to minimize some cost functional  which depends only on the empirical distribution of all other players. The Nash equilibria are characterized by a system of $2N$ equations. According to Lasry and Lions, as $N\to +\infty$, this system of PDEs reduces to the system~\eqref{systemintro} with $\epsilon=1$ where the first equation gives the value function associated to the ``average'' player while the second equation describes the evolution of the distribution of players. The rigorous proof of this limit behaviour has been established by Lasry and Lions in~\cite{ll} for ergodic differential games whereas the evolutive case with nonlocal coupling has been addressed in some recent preprints (\cite{clld,f,la}). 

This approach has been generalized in several directions: long time behaviour (\cite{cllp1,cllp2}), first order systems (\cite{c, cg}), ergodic MFG systems (\cite{be,gps}). For a general overview, we refer the reader to \cite{noteachdou, notecard, gn, notegomes}.

Homogenization of a single PDE (and of systems as well) has been investigated exhaustively.
A summary of the vast literature on this topic is beyond the aim of this paper. Let us only recall from~\cite{blp} that, for a single semilinear equation as the first one in~\eqref{systemintro}, the expansion 
\[u^\epsilon=u^0(x,t)+\epsilon u(x/\epsilon)\]
provides (at least formally) the right ``guess'' for the cell problem (a single ergodic equation) and for the effective Hamilton-Jacobi equation as $\epsilon\to 0^+$. On the other hand, let us recall (see ~\cite{ps}) that, for the homogenization of a Fokker-Planck  equation in the small noise limit as the second equation in~\eqref{systemintro}, the multiplicative formal expansion 
\[m^\epsilon=m^0(t,x)(m(x/\epsilon)+\epsilon m^2(x/\epsilon))\]
takes into account the fact that we expect just weak convergence of the solutions, and provides the right guess for the cell problem and the effective continuity equation.

In this paper we consider  periodic  homogenization of a MFG system under the simultaneous effect of the small noise limit; in other words, we tackle the limit as $\epsilon\to 0$ to the solutions to~\eqref{systemintro}.  This system appears in the limit as $N\to +\infty$ of Nash equilibria for a population of $N$ agents. The dynamics of the $i$-th  agent is  given by 
\[dX^i_t= a_t dt+ \sqrt{\epsilon/2}dW_t\] where $W_t$ is a standard Brownian motion  and $a$ is the control chosen in $\mathbb{R}^n$ in order to minimize the  cost
\[L(x, t, a)= \mathbf{E}\int_0^t \left[\frac{|a_s|^2}{2}+V\left(\frac{X^i_s}{\epsilon}, \sum_{j} \delta_{X_{s}^j}\right)\right]ds,\] 
where we  choose for sake of simplicity the standard affine-quadratic dependence on the control. 
This framework models the case of a differential game which takes place in a  environment with heterogeneities of period $\epsilon$, and with dynamics disturbed at microscopic level by white noise. From a  mathematical point of view,  this scaling (where the viscosity has the same order of magnitude than the size of the heterogeneities) is the most interesting case.  We refer to  Section~\ref{formalexp} for a short discussion of other type of scaling  between the viscosity and the period of the running cost. 
As far as we know, it is the first time that this kind of problems has been considered  in the literature, even if, while completing this work, we became aware of a work in progress on homogenization of MFG system by  P.-L. Lions and P.E.  Souganidis \cite{ls}.

The homogenization limit is interesting as a mathematical question, but may also find applications like e.g. in traffic-flow problems.
Mean field games have widely been used to model traffic flow problems where the cost is the higher the more dense the traffic is. See for example \cite{BDMW},  \cite{CLM}, \cite{LW}  and references therein. 
Changing road conditions - e.g. hills and valleys or a change in the number of lanes-  influences the cost functional, resulting in a spatially varying pre-factor for the cost depending on local traffic density. If the typical distance between cars is much smaller than the scale on which the road conditions vary, then the behavior on a scale which is in turn much larger than the scale of variation of the road conditions could be described by a homogenized mean field game system.
We would like to point out that this observation is only a motivation for the problem under consideration here; applications of homogenization to mean field game systems modeling traffic flow are beyond the scope of this paper.

Because of the small noise limit considered here, the effective system is expected to be a first order system formed by a Hamilton-Jacobi equation and a  transport equation, where the effective Hamiltonian and the effective drift need to be suitably defined. More precisely, in our case, the limit system will be 
\begin{equation}\label{systemeffintro}\begin{cases}
-u_t^0+\bar H(\nabla u^0, m^0)=0, & x\in \mathbb{R}^n, t\in (0,T)\\
m^0_t-{\rm div}(m^0\bar b(\nabla u^0, m^0))=0,& x\in \mathbb{R}^n, t\in (0,T)
\end{cases}\end{equation} with initial/terminal condition $u^0(x,T)=u_0(x)$ and $m^0(x,0)=m_0(x)$.
The effective operators $\bar H(P,\alpha), \bar b(P,\alpha)$  appearing in the limit system are obtained through the following  cell problem. For every  $P\in \mathbb{R}^n$, $\alpha\geq 0$,  find the constant  $\bar H(P,\alpha)$ for which there exists a solution to    the   ergodic mean field game system:  
\begin{equation}\label{ucorintro} \begin{cases}
(i)   -\Delta u     \!+\!\frac{1}{2}|\nabla u  \!+\!P|^2\!-\!V(y, \alpha m  ) =
 \bar  H(P,\alpha), & y\in\mathbb{T}^n\\
(ii)   -\Delta m  -{\rm div}\left(m \left(\nabla u +P\right)\right)=0, & y\in\mathbb{T}^n\\ (iii)  \int_{\mathbb{T}^n} u =0 \qquad  \int_{\mathbb{T}^n} m =1, & \end{cases}
\end{equation} 
while $\bar b$ is given by  \begin{equation}\label{bbar} \bar b(P,\alpha)= \int_{\mathbb{T}^n} (\nabla u +P) m dy,\end{equation} where $(u,m)$ is the  solution to \eqref{ucorintro}.
We observe that, with this choice of small noise and local coupling, the cell problem has a MFG structure. On the other hand, in the case of homogenization of MFG systems without small noise or with nonlocal smoothing term $V$, we expect a cell problem almost decoupled: equation (i) in \eqref{ucorintro} does not depend anymore on $m$, so we can solve it as a standard ergodic equation in $\mathbb{T}^n$. In the case of strong noise, i.e. a second order part of the form $\Delta u$ we expect an explicit formula for the effective Hamiltonian $\bar H(P,\alpha)=\frac{1}{2}|P|^2\!-\int_{\mathbb{T}^n}\!V(y, \alpha)dy$ (see Section~\ref{formalexp}). 

The purpose of this paper is twofold: a study of the effective operators $\bar H$ and $\bar b$ and a convergence result for problem~\eqref{systemintro}. We provide the following properties of $\bar H$ and $\bar b$: local Lipschitz continuity, monotonicity in $\alpha$, coercivity in $P$ of $\bar H$ and asymptotic behaviour of $\bar H$ and $\bar b$ with respect to $P$. As a matter of fact, we establish the following formula
\[\nabla_P \bar H(P, \alpha)=  \bar b(P,\alpha)-\int_{\mathbb{T}^n} V_m(y,\alpha m)\alpha \tilde m m dy \]
where $m$ is the solution to~\eqref{ucorintro} while $\tilde m$ is a function with values in $\mathbb{R}^n$ defined as a solution of a suitable system (see problem~\eqref{mvarP}) and, roughly speaking, coincides with $\nabla_P m$.

The interesting feature is that the limit system \eqref{systemeffintro} may lose the mean field game structure because  $\nabla_P \bar H(P,\alpha)$ may not coincide with $\bar b(P, \alpha)$. We provide an explicit example where this phenomenon appears.

We provide the rigorous convergence result of the solution $(u^\epsilon, m^\epsilon)$ to \eqref{systemintro} to a solution $(u^0, m^0)$ of~\eqref{systemeffintro} respectively strongly in $L^2$ and weakly in $L^1$,  in the following special case: the initial/terminal data are affine for the Hamilton-Jacobi equation and constant for the Fokker-Planck equation, that is
\[ u_0(x)=P\cdot x\qquad m_0(x)\equiv 1, \qquad P\in \mathbb{R}^n.\]
The result is obtained by getting careful apriori estimates on the difference, appropriately rescaled, between the exact solution of the perturbed problem   \eqref{systemintro} and a first order asymptotic expansion. We discuss also a conditional convergence result, under some  additional restrictive conditions.

This paper is organized as follows. In Section \ref{sect:assumpt} we list the standing assumptions and establish the well posedness of system \eqref{systemintro}. In Section \ref{formalexp}, we heuristically provide the effective limit system and we also discuss the cases of alternative asymptotic expansion, nonlocal coupling, and strong noise. In Section \ref{sect:CP}, we define the effective operators, we show their continuity and we provide a variational characterization, coercivity and monotonicity of $\bar H$. Section \ref{sect:regularity} is devoted to local Lipschitz continuity of effective operators and to the computation of $\nabla_P\bar H$. Section \ref{sect:qual} contains the asymptotic behaviour with respect to $P$ and an explicit example where the limit system loses the MFG structure. Finally, in Section \ref{sect:affine}, we prove   the homogenization result for affine constant initial data and in   Section \ref{sect:general} we discuss some further perspectives.

\section{Standing assumptions}\label{sect:assumpt}
In this Section, we collect the assumptions on system~\eqref{systemintro} that will hold throughout all the paper unless it is explicitely stated in a different manner.


\noindent
{\bf Assumptions on the  potential $V$}. \\ We will assume that \[V(y,m):\mathbb{T}^n\times\mathbb{R}\to \mathbb{R}\]
\begin{itemize}\item is a $C^1$ bounded function and w.l.o.g. $V\geq 0$,
\item is monotone increasing with respect to $m$, that is, for every compact interval $K\subset\mathbb{R}$, there exists a positive constant $\gamma_K>0$ such that  
\begin{equation} \label{incr}
V(y, m)-V(y,n)\geq \gamma_K (m-n)\qquad \forall m, n\in K, y\in \mathbb{T}^n.\end{equation}
\end{itemize} 

\noindent
{\bf Assumptions on the initial and terminal data}.\\$u_0\in C^2(\mathbb{T}^n)$. 
$m_0$ is a   smooth nonnegative function on $\mathbb{T}^n$, such that $\int_{\mathbb{T}^n}m_0(x)dx=1$. 

\vskip2mm

We conclude recalling some well known results about existence and uniqueness of the solution to the mean field game system.
\begin{proposition}\label{ex}
For $\epsilon=\frac{1}{k}$ ($k\in\mathbb{N}$), there exists a unique classical solution to  \eqref{systemintro}.
\end{proposition}
\begin{proof}
For the existence, we follow the arguments in \cite[Lemma 4.2]{cllp1}. To this end, it is sufficient to observe that the function $w^\epsilon:=\exp\{-\frac{u^\epsilon}{2\epsilon}\}$ satisfies the linear problem
\[
-w^\epsilon_t-\epsilon \Delta w^\epsilon+\frac{V w^\epsilon}{2\epsilon}=0.
\]
As for uniqueness, we refer to \cite{ll} and \cite[Thm 3.8]{notecard}. 
\end{proof}

\section{Formal asymptotic expansions}\label{formalexp}
In this section we formally derive the cell problem and the effective equation by the method of asymptotic expansions. More precisely, we make the formal ansatz that the solution to the system \eqref{systemintro} satisfies the following asymptotic expansion up to first order in the small parameter~$\epsilon$:
\begin{equation}\label{as}\begin{cases}
u^\epsilon(x,t)=u^0(x,t)+\epsilon u \left(\frac{x}{\epsilon}\right),\\
m^\epsilon(x,t)=m^0(x,t)\left(m \left( \frac{x}{\epsilon}\right)+\epsilon m^2\left(\frac{x}{\epsilon}\right)\right),
\end{cases}\end{equation}
where $m$ and $m^2$ are assumed to be periodic with average respectively $1$ and $0$.

We insert this asymptotic expansion in \eqref{systemintro}, denoting by $P=\nabla u^0$, $\alpha=m^0$ and  $y=\frac{x}{\epsilon}$.

At order $\epsilon^{-1}$  in the second equation, we get 
\begin{equation}\label{corrector} -\Delta_y m -{\rm div}_y(m (P+	\nabla_y u )) =0. \end{equation}

Note that a positive solution to \eqref{corrector}, for $P$, and  $\nabla_y u $ measurable and bounded exists according to \cite[Thm 4.3, pag 136]{bbook}.


Then we collect the terms of order $\epsilon^0$ in the expansion of the $m$- equation 
and  get
\begin{multline}  \label{m2}
m^0(-\Delta_y m^2-{\rm div}_y(m^2(P+\nabla_y u )))\\ = 
-(m m^0)_t-{\rm div}_x(m^0m (P+\nabla_y u ))
+ 2\nabla_x m_0\cdot\nabla_y m. 
\end{multline}
Observe that if $m^2$ is a solution, then  $m^2+km$ is still a solution of the same equation for every $k\in\mathbb{R}$, which allows us to satisfy the mean zero constraint for $m^2.$

The solvability condition for $m^2$ gives (integrating the fast variable $y$ over $\mathbb{T}^n$, while treating the slow variable $x$ as a parameter)
\[m^0_t-{\rm div}_x\left(m^0\left(\int_{\mathbb{T}^n} m (P+\nabla_y u )dy\right)\right)=0\]
which is the expected limit equation for $m^0$. Here we used that $m$ has average equal to $1$ and that the other two terms average to $0$  since $m , u $ are assumed to be periodic in $y$ so the average of $\nabla m  $ is $0$.

Inserting the asymptotic expansion ansatz  in the first equation (the equation for $u$) in \eqref{systemintro}, we get
that the terms of order $\epsilon^0$ are
\[-u^0_t-\Delta_y u    +\frac{1}{2}|\nabla_y u +\nabla_x u^0|^2\!-\!V(y, m^0 m )=0, \]
which gives the formula for the effective Hamiltonian: for every $P\in\mathbb{R}^n$ and $\alpha\geq0$, $\bar H(P,\alpha)$ is
the unique constant for which there exists a periodic solution $u =u(y)$ to
\[-\Delta u  +\frac{1}{2}|\nabla u +P|^2-V(y, \alpha m  )=
\bar  H(P,\alpha),\] and $u_t=\bar H (\nabla u_0,\alpha)$ has to hold.

Summing up, the cell problem is the ergodic mean field game system 
\[ \begin{cases}\label{ergcellu}
 -\Delta u     \!+\!\frac{1}{2}|\nabla u  \!+\!P|^2\!-\!V(y, \alpha m  ) =
 \bar H(P,\alpha)\\
 -\Delta m  -{\rm div}\left(m \left(\nabla u +P\right)\right)=0\\
\int_{\mathbb{T}^n}u=0,\quad \int_{\mathbb{T}^n}m=1. \end{cases}
\]
This cell problem permits to define the effective operators $\bar H(P, \alpha)$ and 
\[\bar b(P,\alpha)= \int_{\mathbb{T}^n} m (\nabla u+P)dy.
\]
So the expected limit system is \eqref{systemeffintro}. 

\subsection{Other asymptotic expansions}

\subsubsection{Time-dependent asymptotic expansion and $\epsilon$-correction to  initial data}\label{timedep}
We also consider a fast time-scale dependence in the asymptotic expansion. This allows for order $\epsilon$-corrections to the  initial/terminal data and leads to a different (time-dependent) cell-problem but the same system of effective equations.
More precisely, consider the ansatz
\begin{equation*}
\begin{cases} u^\epsilon(x,t)=u^0(x,t)+\epsilon u \left(\frac{t}{\epsilon},\frac{x}{\epsilon}\right)\\
m^\epsilon(x,t)=m^0(x,t)\left(m\left(\frac{t}{\epsilon}, \frac{x}{\epsilon}\right)+\epsilon m^2\left(\frac{t}{\epsilon}, \frac{x}{\epsilon}\right)\right).
\end{cases}\end{equation*}
Then, at the highest order (i.e. at order $\epsilon^0$ for the $u$-equation and at order $\epsilon^{-1}$ for the $m$-equation) we get  
the new parabolic cell problem
\begin{equation}\label{paracell}
\begin{cases}
c-u_\tau-\Delta_y u    +\frac{1}{2}|\nabla_y u +P|^2 - V(y, \alpha m)=0\\
 m_\tau-\Delta_y m -{\rm div}_y(m(P+\nabla_y u))=0
\end{cases}
\end{equation}
with $c\equiv- u^0_t$, $P\equiv \nabla_x u^0$ and $\alpha\equiv m^0$.
The effective equation for $u^0$ is determined through solvability conditions for \eqref{paracell}. Note that $u^\epsilon$ will converge   to $u^0$ if 
\begin{equation}\label{ucond}
\lim_{\epsilon\to 0}\epsilon u\left( \frac{t}{\epsilon},\frac{x}{\epsilon}\right)= 0,\quad
\end{equation}
 whereas $m^\epsilon$ converges weakly  to $m^0$ if, at least,  $m$ belongs to some $L^p(\mathbb{T}^n\times (0, +\infty))$ space. Then there exists a unique constant $\bar H(P,\alpha)$ such that \eqref{paracell} admits a solution with these properties and $u^0_t=\bar H(\nabla_x u^0, m^0)$. 

This $\bar H$ is expected to be  the same as the one in \eqref{ergcellu} by the analysis of the long time behaviour of \eqref{paracell}, which is similar to the result in \cite{cllp1}. More precisely, one can expect that  $\epsilon u\left(\frac{t}{\epsilon},y\right) \to (1-t)\lambda$,  uniformly in $y$ and 
$t\in (0,1)$, where $\lambda$ is the unique constant for which there exists a periodic solution to the system 
\[\begin{cases}\label{cllp1}
 \lambda-\Delta_y u    +\frac{1}{2}|\nabla_y u +P|^2 - \widetilde{V}(y, \alpha m)=0\\
-\Delta_y m -{\rm div}(m(P+\nabla_y u))=0,
\end{cases}\]
where $\widetilde V(y, \alpha m):= V(y, \alpha m)-c$. Choosing $c$ as the unique constant for which \eqref{ergcellu} has a solution, by uniqueness of the ergodic constant, there holds $\lambda=0$ (see \cite{cllp1}).  Finally, we get \eqref{ucond} for solutions of \eqref{paracell} and, as $c= u^0_t$, we recover the limiting $u$-equation in \eqref{systemeffintro}.
On the other hand, again by  \cite{cllp1}, we expect   that  $m\left(y,\frac{t}{\epsilon} \right)$ converges weakly to $m(y)$ in some $L^p$ space.

The limit equation in $m$ was found as a solvability condition for $m^2,$ which now should be bounded in $\mathbb{T}au,$ a necessary condition for $m(x/\epsilon,t/\epsilon)$  to converge weakly  in some $L^p$ space. As the equation at order $\epsilon^0$ is now of the form
$$
 m^2_\mathbb{T}au=\Delta_y m^2 +R,
$$
where $R$ contains all nonlinear terms of order zero, we see that the solution can only be periodic in $y$ and bounded as $\mathbb{T}au\to\infty$  if the mean of $R$ vanishes,
which yields the same effective equation as before.

\subsubsection{Case of finite noise}
We consider the mean field game system 
\begin{equation}\label{systemnonsmall}\begin{cases}
-u_t^\epsilon- \Delta u^\epsilon +\frac{1}{2}|Du^\epsilon|^2=V\left(\frac{x}{\epsilon}, m^\epsilon \right), & x\in \mathbb{R}^n, t\in (0,T)\\
m^\epsilon_t- \Delta m^\epsilon-{\rm div}(m^\epsilon\nabla u^\epsilon)=0,& x\in \mathbb{R}^n, t\in (0,T).
\end{cases}\end{equation}
In this case,  the natural  formal ansatz for the solution to the system \eqref{systemnonsmall} is
\begin{equation}\label{as2}
u^\epsilon(x,t)=u^0(x,t)+\epsilon^2 u \left(\frac{x}{\epsilon}\right),\quad
m^\epsilon(x,t)=m^0(x,t)\left(m \left( \frac{x}{\epsilon}\right)+\epsilon^2 m^2\left(\frac{x}{\epsilon}\right)\right).
\end{equation}
We insert this asymptotic expansion in \eqref{systemnonsmall}, denoting by $X=\Delta_x u^0$, $P=\nabla u^0$, $\alpha=m^0$ and  $y=\frac{x}{\epsilon}$.

At order $\epsilon^{-2}$  in the second  equation  we obtain 
$-\Delta m  =0$,  where $m (\cdot)$ is assumed to be    periodic and with mean 1. This implies $m\equiv 1,$ i.e. 
$m^\epsilon(x,t)=m^0(x,t)(1+\epsilon^2 m^2\left(\frac{x}{\epsilon}\right)),$ or, in other words, {\em strong} convergence of $m^\epsilon$ to $m^0$ can be expected.
Inserting the asymptotic expansion in the first equation in \eqref{systemnonsmall}, we get the following cell problem: for every $X\in M_n(\mathbb{R}^n)$, symmetric matrix,  $P\in\mathbb{R}^n$ and $\alpha\geq 0$, $\bar H( X, P,\alpha)$ is
the unique constant for which there exists a periodic solution $u( \cdot)$ to
\[-\Delta u  -{\rm tr} X+\frac{1}{2}|P|^2-V(y, \alpha)=
\bar H( X, P,\alpha).\]
So $\bar H( X,P,\alpha)= -{\rm tr} X+\frac{1}{2}|P|^2-\int_{\mathbb{T}^n}V(y, \alpha)dy$ (see \cite[Ch.2]{blp}).

On the other hand, the  solvability condition for $m^2$ gives the limit equation for $m^0$. 
In conclusion,  the expected effective system is 
\[\begin{cases} -u^0_t-\Delta u^0 +\frac{1}{2}|\nabla u^0|^2-\int_{\mathbb{T}^n}V(y, m^0)dy =0 \\ 
m^0_t-\Delta m^0-{\rm div} \left(m^0\nabla u^0\right)=0.\end{cases}\]
In this case the cell problem decouples, while  the limit problem still has a mean field game structure. Note that the limit coupling
$\bar V(m)=\int_{\mathbb{T}^n} V(y,m)dy$ satisfies \eqref{incr}. Hence,
the effective system fulfills uniqueness of solutions by a well known argument (see \cite{ll}).

  \subsubsection{Case of non local coupling}
Also in the case of non local coupling we expect that the cell problem decouples. 
We consider the following example to illustrate this issue (see \cite[Rem 2.10]{notecard}). We define  $\mathcal{L}(m)=w$ as    the periodic  solution to  
$-\Delta w=m-\int_{\mathbb{T}^n} m(x)dx$  with $w(0)=0$ and we consider 
\begin{equation}\label{systemnonlocal}\begin{cases}
-u_t^\epsilon- \epsilon\Delta u^\epsilon +\frac{1}{2}|Du^\epsilon|^2=V\left(\frac{x}{\epsilon},\mathcal{L}(m^\epsilon(\cdot,t))\right), & x\in \mathbb{T}^n, t\in (0,T)\\
m^\epsilon_t- \epsilon\Delta m^\epsilon-{\rm div}(m^\epsilon\nabla u^\epsilon)=0,& x\in \mathbb{T}^n, t\in (0,T).
\end{cases}\end{equation}

We consider the formal asymptotic expansion \eqref{as}. We add  the ansatz that also $w^\epsilon=\mathcal{L}(m^\epsilon)$ satisfies
the  asymptotic expansion which is the standard expansion for homogenization of second order equations 
\[w^\epsilon(x,t)=w^0(x,t)+\epsilon^2 w\left(\frac{x}{\epsilon}\right),\]
where $-\Delta w^0=m^0-1$, 
and  $w$ is the  periodic function with zero average such that $-\Delta w (y)= m^0(x)(m(y)-1)$. This implies that in the cell problem $m$ does not appear anymore in the first equation. So, the cell system is  
\[\begin{cases} 
 -\Delta u     \!+\!\frac{1}{2}|\nabla u  \!+\!P|^2\!-\!V(y, \alpha) =
 \bar  H(P,\alpha)\\
 -\Delta m  -{\rm div}\left(m \left(\nabla u +P\right)\right)=0,\\   \int_{\mathbb{T}^n} u =0 \qquad  \int_{\mathbb{T}^n} m =1, \end{cases} 
\] where $P=\nabla_x u^0$ and $\alpha=w^0$. 

In this case the expected limit system is given by 
\[\begin{cases}
-u_t^0+\bar H(\nabla u^0, \mathcal{L}(m^0(\cdot, t)))=0, & x\in \mathbb{R}^n, t\in (0,T)\\
m^0_t-{\rm div}(m^0\bar b(\nabla u^0, m^0))=0,& x\in \mathbb{R}^n, t\in (0,T).
\end{cases}\] 


\section{Cell problem, effective Hamiltonian and effective vector field.}\label{sect:CP}
This Section is devoted to the definition and the properties of the effective operators. In the first Proposition, we tackle the solvability of the cell problem obtained by the formal asymptotic expansion. 

\begin{proposition}\label{cellprop} Under the standing assumptions, 
for every  $P\in\mathbb{R}^n$ and $\alpha\geq  0$, consider the system
\begin{equation}\label{ucor} \begin{cases}
(i)  &  -\Delta u     \!+\!\frac{1}{2}|\nabla u  \!+\!P|^2\!-\!V(y, \alpha m  ) =
 \bar  H(P,\alpha) \qquad x\in \mathbb{T}^n\\
(ii)  & -\Delta m  -{\rm div}\left(m \left(\nabla u +P\right)\right)=0 \qquad x\in \mathbb{T}^n,\\ (iii) &  \int_{\mathbb{T}^n} u =0 \qquad  \int_{\mathbb{T}^n} m =1. \end{cases}
\end{equation}
Then there exists  a  unique constant $\bar  H(P,\alpha)$ such that the system 
admits a solution  $(u,m)$. Moreover this solution is unique and $u\in C^{2,\gamma}$, $m\in W^{1,p}$ for all $\gamma \in (0,1)$ and all $p>1$.

Finally there exists a constant $c>0$ such that $m\geq c>0$.

 \end{proposition}   \begin{proof}
 Observe that if $\alpha=0$, in the first equation there is no $m$-dependence. So, we solve the ergodic problem (i) in \eqref{ucor} (see \cite{bf87}). Then we plug the solution $u$ in the second equation (ii), getting a solution $m$ (\cite{bbook}), with the  desired properties.

So, let assume that $\alpha>0$. We use a similar argument as in \cite{cm}. We consider $K:=\{m\in C^{0,\gamma}(\mathbb{T}^n)\mid m\textrm{ Lebesgue-density\ of\ a measure}\}$; observe that $K$ is a closed subset of $C^{0,\gamma}(\mathbb{T}^n)$. We introduce an operator $T:K\to K$, $m'\mapsto m$ as follows: given $m'\in K$, we solve the ergodic equation \eqref{ucor}-(i) with $m$ replaced by $m'$ coupled with the first condition in \eqref{ucor}-(iii). Then, we get $m$ solving \eqref{ucor}-(ii) coupled with the second condition in \eqref{ucor}-(iii).

Let us prove that $T$ is well-posed. By \cite{bf87}, there exists a unique pair $(u,\bar  H)\in W^{2,p}(\mathbb{T}^n)\times \mathbb{R}$ for any $p\geq 1$ ($u$ unique up to an additive constant) solving \eqref{ucor}-(i). In particular, $\nabla u$ is H\"older continuous. By \cite[Thm4.2]{bbook}, problem \eqref{ucor}-(ii) admits a solution $m\in W^{1,p}(\mathbb{T}^n)$ which is unique up to multiplicative constants. Hence, the mapping $T$ is well defined.

Let us now prove the $T$ is continuous and compact. To this end, we recall from \cite{bf87} that there holds
\[
|\bar  H|\leq K+|P|^2 \qquad \|u\|_{W^{2,p}(\mathbb{T}^n)}\leq C
\]
where $K$ depends on $\|V\|_\infty$ and  $C$ is a constant depending only on $\|V\|_\infty$, $P$ and $p$. On the other hand, we also recall from \cite[Thm4.2]{bbook} and \cite[Lemma 2.3]{be} that there holds
\[
\|m\|_{W^{1,p}(\mathbb{T}^n)}\leq C'
\]
where $C'$ is a constant depending only on $\|\nabla u\|_\infty$, and so in particular only from $\|V\|_\infty$ and  $P$.

Consider a sequence $\{m'_n\}$, with $m'_n\in K$ and $m'_n\to m'$ in the $C^{0,\gamma}$-topology. Therefore, the corresponding solutions $(u'_n, \bar  H_n)$ to  \eqref{ucor}-(i) with $m$ replaced by $m_n'$ coupled with the first condition in \eqref{ucor}-(iii) are uniformly bounded in~$W^{2,p}(\mathbb{T}^n)\times \mathbb{R}$. Consequently, the corresponding solutions $m_n$ to \eqref{ucor}-(ii) coupled with the second condition in \eqref{ucor}-(iii) are uniformly bounded in $W^{1,p}(\mathbb{T}^n)$.
So, by possibly passing to a subsequence, we get $u'_n\to u^*$ and $\bar  H_n\to \bar  H^*$. By stability and by uniqueness of problem \eqref{ucor}-(i), we get $u^*=u$ and $\bar  H^*=\bar  H$. In a similar manner, we get $m_n\to m$. Hence, we have accomplished the proof of the continuity of $T$.

The map $T$ is also compact because, $m\in W^{1,p}(\mathbb{T}^n)$ for any $p\geq 1$; in particular, $m\in C^{0,\alpha'}(\mathbb{T}^n)$ for any $\alpha'\in (0,1)$. By the compactness of the embedding $C^{0,\gamma'}(\mathbb{T}^n)\to C^{0,\gamma}(\mathbb{T}^n)$ with $\gamma'>\gamma$, we infer the compactness of $T$.

We conclude by Schauder's fixed point theorem the proof of the existence of a solution $(u,m)\in W^{2,p}(\mathbb{T}^n)\times W^{1,p}(\mathbb{T}^n)$ to \eqref{ucor} for any $p>1$. By a standard bootstrap argument, we obtain the claimed regularity, i.e. $u\in C^{2,\gamma}$. 

The last statement is proved in  \cite[Thm 4.3, pag 136]{bbook}. 
\end{proof}

We prove now some properties of the effective Hamiltonian $\bar H$ and the effective vector field $\bar b$, defined respectively in~\eqref{ucor} and~\eqref{bbar}.
\begin{proposition}\label{4.2}
$\bar H$ is coercive in $P$, that is \[ \frac{|P|^2}{2}-\|V\|_\infty \leq \bar H(P,\alpha)\leq \frac{|P|^2}{2}\]
and  is decreasing in $\alpha$.

Moreover, for any $\gamma\in (0,1)$, $p\in(1,+\infty)$, the maps \[(P,\alpha)\to\bar H(P,\alpha)\in\mathbb{R},\qquad(P,\alpha)\to \bar b(P,\alpha)\in \mathbb{R},\qquad (P,\alpha)\to (u,m)\in C^{1,\gamma}\times W^{1,p}\] are all continuous, where   $(u,m)$ is the solution to \eqref{ucor}.

\end{proposition}

\begin{proof}
{\bf Coercivity.} Let $y\in\mathbb{T}^n$ be a maximum point of $u,$ which is a solution to equation (i) in \eqref{ucor}. Then \[\bar H(P, \alpha)\geq \frac{|P|^2}{2}-\|V\|_\infty\] which implies the desired coercivity.
Moreover if we compute the equation (i) in  \eqref{ucor} at a minimum point of $u$, recalling that $V\geq 0$, we get \[\bar H(P,\alpha)\leq \frac{|P|^2}{2}.\]
{\bf Monotonicity.} Denote  $m^{\alpha_1}=\alpha_1 m^1$, where $(u^1,m^1)$ solves (ii) in \eqref{ucor} with $(P,\alpha_1)$ and $m^{\alpha_2}=\alpha_2 m^2$, where $(u^2,m^2)$ solves (ii) in  \eqref{ucor} with $(P,\alpha_2)$.
Let $\bar u=u^1-u^2$ and $\bar m=m^{\alpha_1}-m^{\alpha_2}$. Then $(\bar u, \bar m)$ solves
\[\begin{cases} -\Delta\bar u +\frac{1}{2}|\nabla u^1+P|^2-\frac{1}{2}|\nabla u^2+P|^2 -V(y,m^{\alpha_1})+V(y,m^{\alpha_2})\\ \qquad =\bar H(P,\alpha_1)-\bar H (P, \alpha_2)\\ -\Delta \bar m-{\rm div}(m^{\alpha_1}(\nabla u^1+P)-m^{\alpha_2}(\nabla u^2+P))=0 \\ \int_{\mathbb{T}^n} \bar u=0 \ \ \int_{\mathbb{T}^n} \bar m=\alpha_1-\alpha_2.
\end{cases}\]
We multiply the first equation by $\bar m$ and the second by $\bar u$,  we integrate and we subtract one equation from the other to obtain
\[(\bar H(P,\alpha_1)-\bar H (P, \alpha_2))(\alpha_1-\alpha_2)=-\int (V(y,m^{\alpha_1})-V(y,m^{\alpha_2}))(m^{\alpha_1}-m^{\alpha_2})+\]\[+\int \bar m\left(\frac{1}{2}|\nabla u^1+P|^2-\frac{1}{2}|\nabla u^2+P|^2\right)-\nabla\bar u\cdot (m^{\alpha_1}(\nabla u^1+P)-m^{\alpha_2}(\nabla u^2+P)).\]
It is easy to check that if $q_1, q_2\in\mathbb{R}^n$ and $n_1,n_2\geq 0$ then \[(n_1-n_2)\left(\frac{|q_1|^2}{2}-\frac{|q_2|^2}{2}\right)- (q_1-q_2)\cdot (n_1 q_1-n_2 q_2)= -\frac{n_1+n_2}{2} |q_1-q_2|^2\leq 0.\]
Then applying this equality to $q_1= \nabla u^1+P$, $q_2=\nabla u^2+P$, $n_1=m^{\alpha_1}$, $n_2=m^{\alpha_2}$, we get that \[ (\bar H(P,\alpha_1)-\bar H (P, \alpha_2))(\alpha_1-\alpha_2)\leq -\int (V(y,m^{\alpha_1})-V(y,m^{\alpha_2}))(m^{\alpha_1}-m^{\alpha_2}) \leq 0.\]
{\bf Continuity.} Consider a sequence $(P_n,\alpha_n)\to (P,\alpha)$.  Observe that $\bar H(P_n, \alpha_n)$ are uniformly bounded in $n$, due to coerciveness. For every $n$, let the pair $(u_n,m_n)$ be the solution to the cell problem \eqref{ucor} with  $(P,\alpha)$ replaced by $(P_n, \alpha_n)$.  By   the same argument as in \cite{bf87}, we get a priori bounds on $\nabla u_n$ independent of $n$, i.e. there exists $C>0$ such that $|\nabla u_n|\leq C$ for every $n$. So, up to a subsequence, $\nabla u_n\to v$ weakly$^*$ in $L^{\infty}$. Moreover, $\|m_n\|_{W^{1, p}} \leq K$ for every $p\in (1,\infty)$ and every $n$ (see \cite[Lemma 2.3]{be}). So, by Morrey's inequality, for $p$ sufficiently large, there exists a subsequence $m_n$ converging in $C^{0,\gamma}$ to $m$. So, up to extracting a converging sequence, we get, by stability of viscosity solutions, that $u_n\to u$, where $u$ is a viscosity solution to \[ -\Delta u     \!+\!\frac{1}{2}|\nabla u  \!+\!P|^2\!-\!V(y, \alpha m) =
 \lim_n\bar  H(P_n,\alpha_n).\] Moreover $u$ is Lipschitz, and  by a standard argument therefore even smooth.  
This implies that $v=\nabla u$, and then by uniqueness, along any converging subsequences, $\lim_n\bar  H(P_n,\alpha_n)=\bar H(P, \alpha)$, $\lim_n u_n=u$ uniformly and $\lim_n m_n=m$ uniformly and in $W^{1,p}$, for $p\in (1, +\infty)$, where $(u,m)$ is the solution to \eqref{ucor}. By standard elliptic regularity theory and a priori bounds on the solution to \eqref{ucor}-$(i)$ with uniformly Holder continuous source term, see also Proposition \ref{cellprop},  the convergence of $u_n$ is also in $C^{1,\gamma}$.
\end{proof}

We conclude recalling a variational characterization of the effective Hamiltonian.
	
Following \cite{ll}, we introduce the following energy  functional in $H^1(\mathbb{T}^n)\times H^1(\mathbb{T}^n)$ for every fixed $P\in \mathbb{R}^n$ and $\alpha\geq 0$:
\begin{equation}\label{energy}
E_{P,\alpha}(v,n)=\int_{\mathbb{T}^n}  n\frac{|\nabla v+P|^2}{2}+ \nabla n\cdot (\nabla v+P)-\Phi_\alpha(y,  n) dy\end{equation}
where $(\Phi_\alpha)_n\equiv \frac{\partial \Phi_\alpha}{\partial n}(y,n)=V(y,\alpha n)$.
\begin{lemma}
\[\frac{\partial E_{P,\alpha}}{\partial m}(u,m)=0 \quad \text{  iff $(u,m)$ solves (i) in  \eqref{ucor}}\]
\[\frac{\partial E_{P,\alpha}}{\partial u}(u,m)=0\quad \text{ iff $(u,m)$ solves (ii) in  \eqref{ucor}}.\]
Moreover, for $(u,m)$ solution to \eqref{ucor}, there holds true \begin{equation}\label{effh} \bar H(P,\alpha)= E_{P,\alpha}(u ,m )+\int_{\mathbb{T}^n} \left(\Phi_\alpha(y, m )-V(y,\alpha m )  m \right) dy.\end{equation}
 \end{lemma}\begin{proof} The first statement is a straightforward computation. We give a quick sketch of it.
Fix $m$ and let $u$ be the solution to equation (i) in  \eqref{ucor} with this fixed $m$. Then, for every $\epsilon>0$ small and smooth $\phi$ with $\int_{\mathbb{T}^n}\phi =0$ we get
\begin{eqnarray*}\frac{ E_{P,\alpha}(u,m+\epsilon \phi)- E_{P,\alpha}(u,m)}{\epsilon}
& =& \int_{\mathbb{T}^n} \phi\frac{|\nabla u+P|^2}{2}+ \nabla \phi\cdot (\nabla u+P)-\\
&-& \int_{\mathbb{T}^n}\frac{\Phi_\alpha(y,  m+\epsilon\phi)-\Phi_\alpha(y,  m )}{\epsilon} .\end{eqnarray*}
So, letting $\epsilon\to 0$, and using the fact that $u$ solves (i) in \eqref{ucor} we get
\[\lim_{\epsilon\to 0}\frac{ E_{P,\alpha}(u,m+\epsilon \phi)- E_{P,\alpha}(u,m)}{\epsilon}= \bar H(P, \alpha)\int_{\mathbb{T}^n}\phi =0.\] The other implication is obtained reverting this argument.
Analogous arguments work for the second statement.

Finally, in order to prove~\eqref{effh}, we first observe that multiplying equation (ii) in~\eqref{ucor} by $u$ and integrating over $\mathbb{T}^n$ leads to
\[ \int \Delta u \  m=\int m  \left(\nabla u  +P\right)\nabla u
.\]
Now, multiplying equation (i) in~\eqref{ucor} by $m$ and integrating over $\mathbb{T}^n$, we infer~\eqref{effh}.\end{proof}

\section{Regularity properties of the effective operators}\label{sect:regularity}
In this section we study the relation between the effective Hamiltonian and the effective vector field. The main result is the computation of $\nabla_P \bar H(P,\alpha)$ established in Theorem \ref{hbar}. 
To do this, we first of all derive local Lipschitz estimates of $\bar H$. Finally we will provide some related regularity results also for the vector field $\bar b$. 

\subsection{Variation of $\bar H$ with respect to $P$}

Fix now $\delta$ small and consider for every $i=1,\dots, n$, the cell problem \eqref{ucor} associated to $(P+\delta e_i, \alpha)$ where $\{e_i\}_{i=1,\dots,n}$ is an orthonormal basis of $\mathbb{R}^n$. We denote with $(u^\delta_i, m^\delta_i)$ the solution.
First of all by Proposition \ref{4.2},  \begin{equation}\label{limdelta}\lim_{\delta\to 0} u^\delta_i= u \ \text{ in }C^{1,\gamma}\qquad\lim_{\delta\to 0} m^\delta_i= m \ \ \text{ in }C^{0} \end{equation} where $(u,m)$ is the solution to \eqref{ucor} associated to $(P,\alpha)$. 


Our aim is to characterize the following functions, and compute their limits as $\delta\to 0^+$
\begin{equation}\label{deltaw} w_i^\delta(x):= \frac{u^\delta_i-u}{\delta}\qquad n_i^\delta=\frac{m^\delta_i-m}{\delta}. \end{equation}
Note that $(w_i^\delta, n_i^\delta)$ is a solution of the following system
\begin{equation}\label{mdelta}
\begin{cases} (i) &  -\Delta  w_i^\delta +  (\nabla w^\delta_i+e_i)\cdot \frac{2P+\delta e_i +\nabla u+\nabla u^\delta_i }{2}- V_m(y,\alpha \widetilde n^\delta)\alpha  n_i^\delta\\ &  =\frac{\bar H(P+\delta e_i,\alpha)-\bar H(P,\alpha)}{\delta}\\ (ii)&
-\Delta n^\delta_i-{\rm div}\big((P+\delta e_i+\nabla u_i^\delta)  n_i^\delta \big) ={\rm div }(m(\nabla   w_i^\delta+ e_i)) \\ (iii) 
& \int_{\mathbb{T}^n} n_i^\delta=\int_{\mathbb{T}^n} w_i^\delta=0 \end{cases}
\end{equation}
where we used the mean value theorem to write \begin{equation}\label{mean}\frac{1}{\delta} (V(y,\alpha m^\delta_i)-V(y,\alpha m))= V_m(y,\alpha {\widetilde{ n}}^\delta)\alpha  n_i^\delta \end{equation} for some $\widetilde n^\delta$ such that $n^\delta(y)\in (m(y),m^\delta_i(y))$.

First of all we prove a priori estimates on $w^\delta_i, n^\delta_i$ and that, as $\delta\to 0$, then the right hand side of (i) in \eqref{mdelta} does not explode. 
\begin{proposition} \label{estidelta} Let $\alpha>0$. 
There is a constant $C$ depending on  $(P, \alpha)$, 
such that, for any $\delta$ sufficiently small,  
\[\|n_i^\delta\|_2^2+\|\nabla w_i^\delta\|_2^2\le C\]
and \begin{equation}\label{hbarlip}\frac{\bar H(P+\delta e_i,\alpha)-\bar H(P,\alpha)}{\delta}\leq C.\end{equation} 
This implies in particular that $\bar H$ is locally Lipschitz in $P$. 
\end{proposition} 
\begin{proof}
We set $f:=\nabla u+P+\delta e_i$. Note that $\|f\|_2^2$ is bounded uniformly in $\delta$ by Proposition~\ref{cellprop}. Hence, we can write (i) and~(ii) of~\eqref{mdelta} as
\begin{equation*}
\begin{cases}
(i)&-\Delta w^\delta_i +(\nabla w^\delta_i+e_i)\cdot f +\frac{\delta}{2}(|\nabla w^\delta_i|^2-1)-V_m(y,\alpha \widetilde{n}^{ \delta})\alpha n^{ \delta}_i\\ &   =\frac{\bar H(P+\delta e_i,\alpha)-\bar H(P,\alpha)}\delta\\
(ii)&
-\Delta n^\delta_i -{\rm div}\big((f+\delta \nabla w^\delta_i)n^\delta_i\big)={\rm div}\big(m(\nabla w^\delta_i+e_i)\big)
\end{cases}
\end{equation*}
Now we test (i) with $n_i^\delta$ and (ii) with $w_i^\delta.$ Note that $n_i^\delta$ has mean zero, so the
constant term on the right hand side of (i) drops out. We subtract and get
\begin{multline*}
\int\left((e_i\cdot f)n_i^\delta+\frac{\delta}{2}(|\nabla w^\delta_i|^2-1)n_i^\delta -V_m(y, \alpha \widetilde n^\delta)\alpha(n_i^\delta)^2- 
\delta|\nabla w^\delta_i|^2n_i^\delta  \right)\\=\int\left( m|\nabla w_i^\delta|^2+(\nabla w_i^\delta\cdot e_i)m\right)
\end{multline*}i.e., recalling that $\int n_i^\delta=0$, 
$$
\int\left((e_i\cdot f)n_i^\delta-(\nabla w_i^\delta \cdot e_i)m\right)=\int\left(V_m(y,\alpha \widetilde n^\delta)\alpha(n_i^\delta)^2+\left(\frac{\delta}{2}n_i^\delta+m\right)|\nabla w^\delta_i|^2\right).
$$
Since $m$ and $m_i^\delta$ are bounded in $L^\infty$ uniformly in $\delta$ (i.e. no concentrations phenomena appear), then our assumptions on $V$ (monotone with derivative bounded away from zero) allow to estimate $V_m(\cdot)>\gamma_K>0$ for some constant $\gamma_K$ depending on  $m$ (by \eqref{incr}, \eqref{limdelta} and \eqref{mean}) and moreover $\frac{\delta}{2}n_i^\delta+m=\frac{1}{2}\left(m_i^\delta +m\right)>c>0$ by Proposition \ref{cellprop} and \eqref{limdelta}. 
So the right hand side is $\ge c(\|n_i^\delta\|_2^2+\|\nabla w_i^\delta\|_2^2)$ and we conclude with Young's  inequality that there exists
a constant $C$ which depends only on a priori estimates of the correctors, on $\alpha$ and $P$ 
such that$$
\|n_i^\delta\|_2^2+\|\nabla w_i^\delta\|_2^2\le C(\nabla u, m).
$$
This together with testing $(i)$ with the constant 1 gives \eqref{hbarlip}.

\end{proof}

In order to characterize the limit as $\delta\to 0$ of  $w_i^\delta, n_i^\delta$, solution to \eqref{mdelta},  we introduce  an auxiliary system and we study existence and uniqueness of its solutions.

\begin{lemma}\label{lemmavar}
Let $(u,m)$ the solution to \eqref{ucor}, with $\alpha>0$.
Then for every $i=1,\dots, n$  there  exists a unique $c_i(P,\alpha)\in \mathbb{R}$ such that there exists a solution $(\tilde u_i, \tilde m_i)$ to the system
\begin{equation}\label{mvarP}
\begin{cases} (i) &  -\Delta\tilde u_i + \nabla\tilde u_i\cdot (\nabla u+P) + (\nabla u+P)\cdot e_i -V_m(y,\alpha m)\alpha\tilde m_i=c_i(P,\alpha)\\ (ii)&
-\Delta \tilde m_i-{\rm div}\big((P+\nabla u)\tilde m_i \big) ={\rm div }(m(\nabla \tilde u_i+ e_i)) \\
& \int_{\mathbb{T}^n}\tilde m_i=\int_{\mathbb{T}^n}\tilde u_i=0. \end{cases}
\end{equation} Moreover the solution is unique and smooth, i.e. $(\tilde u_i, \tilde m_i)	\in C^{2,\gamma}\times  W^{1,p}$ for all $\gamma \in (0,1)$ and all $p>1$.
\end{lemma} \begin{proof}
Let us observe that the adjoint operator to the one in (i) coincides with the operator in (ii) of \eqref{ucor} which has a $1$-dimensional kernel (see \cite[Thm4.3]{bbook}), so the compatibility condition for the existence of a solution to (i) reads as follows
\begin{equation}\label{ci} c_i(P,\alpha)= \int_{\mathbb{T}^n} \left((\nabla u+P)_i m -V_m(y,\alpha m)\alpha \tilde m_i m\right)dy. \end{equation}

For the uniqueness of the solution of the system, the argument is quite standard. Let $(\tilde u_i, \tilde m_i, c_i(P,\alpha))$ and $(\tilde v_i, \tilde n_i, k_i(P,\alpha))$ be two solutions to \eqref{mvarP}. Let $\tilde u=\tilde u_i-\tilde v_i$ and $\tilde m=\tilde m_i-\tilde n_i$. Then $\tilde u, \tilde m$ solve the following system:
\[\begin{cases} (i) &  -\Delta\tilde u + \nabla\tilde u\cdot \nabla u   -V_m(y,\alpha m)\alpha\tilde m=c_i(P,\alpha)-k_i(P,\alpha)\\ (ii)&
-\Delta \tilde m-{\rm div}\big(\nabla u\tilde m  \big) ={\rm div }(m \nabla \tilde u) \\
& \int_{\mathbb{T}^n}\tilde m =\int_{\mathbb{T}^n}\tilde u =0. \end{cases}\]
We multiply (i) by $\tilde m$ and (ii) by $\tilde u$, subtract equation (ii) from (i) and integrate on the torus $\mathbb{T}^n$: so we obtain, after some easy integrations by parts,
\[-\int_{\mathbb{T}^n}\left(V_m(y,\alpha m)\alpha\tilde m^2 + |\nabla u|^2 m \right)=0.\]
Due to   assumption\eqref{incr} on $V$ and on the fact that $m>0$, both terms in the previous integral are positive. This implies that $\tilde m=0$, so $\tilde m_i=\tilde n_i$, and moreover $c_i(P,\alpha)=k_i(P, \alpha)$ and $u_i=v_i$.

For the existence of a solution, we argue by standard fixed point argument (see e.g. \cite{be}, \cite{notecard}). We sketch briefly the argument. First of all note that both equations in
\eqref{mvarP} are linear, with coefficients in $C^{0,\gamma}$ (due to our assumptions on $V$ and to Proposition \ref{cellprop}).

Fix now  $\bar n\in C^{0,\gamma}(\mathbb{T}^n)$, with $\int_{\mathbb{T}^n}\bar  n=0$, and solve equation (i) with this fixed $\bar n$ in place of $\tilde m_i$. We obtain that there exists a unique constant $c_i^{\bar n}(P,\alpha)$, given by \eqref{ci}, for which the equation admits a   solution $v$. This solution is unique, by the constraint on the average, and smooth, say in $C^{2,\gamma}$ (since it is Lipschitz, and then we apply standard elliptic regularity theory).
Now, we replace $m$ with $v$ in equation (ii) in \eqref{mvarP} and solve it. We get that there  exists a unique solution $l\in W^{1,p}$ for every $p>1$, with the constraint that $\int_{\mathbb{T}^n}l =0$. Indeed the existence of a one parameter family of solutions to (ii) in $W^{1,2}$ is obtained by Fredholm alternative (see e.g. \cite{bbook}),  and uniqueness is obtained adding the constraint on the average. The enhanced regularity can be obtained as in \cite[Lemma2.3]{be}.

So, we constructed a map \begin{equation}\label{t} T: \mathcal{B}:=\left\{\bar n\in C^{0,\gamma} \ |\ \int_{\mathbb{T}^n} \bar n=0\right\}\longrightarrow
\mathcal{B}\end{equation} such that $T:\bar n\to  (v, c_i^n(P,\alpha))\to l$. The continuity of such map can be obtained as in \cite[Thm 3.1]{notecard}. Let $m_n$ be a sequence in $\mathcal{B}$ converging uniformly to $\bar n$; let $u_n$ be the solution to (i) with $m_n$ and $v$ the solution to (i) with $n$. Then $c_i^{m_n}(P, \alpha)\to c_i^{\bar n}(P, \alpha)$ and  $V_m(y,\alpha m)\alpha m_n\to V_m(y,\alpha m)\alpha\bar n$ uniformly. By stability of viscosity solutions, we get that $u_n\to v$ uniformly. Moreover $\nabla u_n$ are uniformly bounded in $C^{0,\gamma}$ (due to standard elliptic regularity theory and uniform convergence of the coefficients of equation (i), see also \cite[Lemma 2.2]{be}) so we can extract a subsequence  $\nabla u_n\to \nabla v$ uniformly. Let $\mu_n$ and $\nu$ be the solutions to (ii) with $\nabla u$ replaced respectively by $\nabla u_n$ and $\nabla v$, so $\mu_m=T(m_n)$ and $\nu=T(\bar n)$. By the $L^\infty$ uniform bound on $\nabla u_n$, we get that $\mu_n$ are uniformly bounded in $W^{1,p}$ for every $p>1$ (see \cite[Lemma 2.3]{be}) and then by Sobolev embedding, they are uniformly bounded in  $C^{0,\gamma}$. Passing to a converging subsequence, we get that $\mu_n\to l$ uniformly, and moreover $l$ is a weak solution to (ii), with $\nabla v$. By uniqueness we conclude that  $l=\nu$, moreover again by uniqueness of the limits, we have convergence for the full sequence. This gives continuity of the operator $T$.
Compactness can be obtained as in \cite[Thm 2.1]{be}. This allows to conclude by Schauder's fixed point theorem.
\end{proof}

\begin{theorem} \label{hbar} Let $\alpha>0$. 
For every $i=1,\dots,n$,
\[\lim_{\delta\to 0} \frac{\bar H(P+\delta e_i, \alpha)-\bar H(P, \alpha)}{\delta}=  \bar b_i(P,\alpha)-\int_{\mathbb{T}^n} V_m(y,\alpha m)\alpha \tilde m_i m dy  \]
 where $m$ is the solution to (ii) in \eqref{ucor} and $\tilde m_i $ is the solution to (ii) in \eqref{mvarP}.
\end{theorem}
\begin{proof}
Note that the coefficients of the system \eqref{mdelta}, due to \eqref{limdelta}, are converging to the coefficients of the system \eqref{mvarP}. Moreover $\frac{\bar H(P+\delta e_i, \alpha)-\bar H(P,\alpha)}{\delta}$ is bounded by \eqref{hbarlip}, then up to a subsequence, we can assume it is converging to some constant.  Moreover due to  the apriori bounds in Proposition \ref{estidelta}, we can extract subsequences $\nabla w^\delta_i, n^\delta_i$ converging weakly in $L^2$.  By stability and by uniqueness of the solution to \eqref{mvarP}, we get the convergence of $w^\delta_i, n^\delta_i$ to $\tilde u_i, \tilde m_i$ and  by uniqueness of the constant $c_i$ for which the system \eqref{mvarP} admits a solution we get that \[c_i(P,\alpha)=\lim_{\delta\to 0}\frac{\bar H(P+\delta e_i, \alpha)-\bar H(P,\alpha)}{\delta}\] which gives the desired conclusion, recalling formula \eqref{ci}.

\end{proof}
\subsection{Variation of $\bar H$ with respect to $\alpha$}
We shall proceed as in the previous section to compute the variation of $\bar H$ with respect to $\alpha$. 
\begin{lemma}\label{estialpha} Let $\alpha>0$. Then 
 $\bar H$ is locally Lipschitz in $\alpha$, i.e. there is a constant $C$ depending on  $P, \alpha$,  
such that, for any $\delta$ sufficiently small,  
\[\frac{\bar H(P,\alpha+\delta)-\bar H(P,\alpha)}{\delta}\leq C.\] 
\end{lemma}
\begin{proof} 
The proof is similar to that in Proposition \ref{estidelta}.

For $\delta$ small we consider the solution $(u^\delta, m^\delta)$ to the cell problem \eqref{ucor} associated to $(P, \alpha+\delta)$. By Proposition \ref{4.2},  we get that as $\delta\to 0$, $u^\delta\to u$ in $C^{1,\gamma}$, $m^\delta\to m$ in $C^0$. The functions \[w^\delta=\frac{u^\delta-u}{\delta}\quad n^\delta=\frac{m^\delta-m}{\delta}\] fulfill
\begin{equation}\label{mdeltaalpha}
\begin{cases} (i) &  -\Delta  w^\delta +  \delta \frac{|\nabla w^\delta|^2}{2}+ \nabla w^\delta\cdot (\nabla u+P)- V_m(y,\alpha \tilde{n}^\delta) (\alpha  n^\delta+m^\delta)\\ & =\frac{\bar H(P, \alpha+\delta)-\bar H(P,\alpha)}{\delta}\\ (ii)&
-\Delta n^\delta-{\rm div}\big((P+\nabla u^\delta)  n^\delta \big) ={\rm div }(m(\nabla   w^\delta)) \\
(iii) & \int_{\mathbb{T}^n} n^\delta=\int_{\mathbb{T}^n} w^\delta=0 \end{cases}
\end{equation}
for some $\tilde{n}^\delta(y)\in (m(y),(1+\frac{\delta}{\alpha}) m^\delta(y))$. 
We multiply the first equation by $n^\delta$ and the second by $w^\delta$, we subtract the second equation from the first and recalling that $w^\delta,n^\delta$ have mean zero, we get
\[\int \frac{(m^\delta+m)}{2}|\nabla w^\delta|^2+ V_m(y,\alpha \tilde{n}^\delta)  \alpha  (n^\delta)^2=-\int  V_m(y,\alpha \tilde{n}^\delta) n^\delta m^\delta. \]
By the Young inequality, this implies  \[\int \frac{(m^\delta+m)}{2}|\nabla w^\delta|^2+ \frac{1}{2}V_m(y,\alpha \tilde{n}^\delta)  \alpha  (n^\delta)^2\leq \frac{1}{2\alpha}\int  V_m(y,\alpha \tilde{n}^\delta) (m^\delta)^2  \] which in particular, recalling that $V_m\geq \gamma_K$ and $m^\delta+m>0$, implies $\|n^\delta\|_2^2+\|\nabla w^\delta\|_2^2\le C$, with a constant depending on $m, u, P, \alpha$. 

By testing the first equation in \eqref{mdeltaalpha} by $1$ and integrating over~$\mathbb{T}^n$
these  bounds imply the desired local Lipschitz continuity.
\end{proof}
\begin{lemma}\label{lemmavaralpha}
Let $(u,m)$ the solution to \eqref{ucor}.
Then  there  exists a unique $k(P,\alpha)$ such that there exists a solution $(\bar u, \bar m)$ to the system
\begin{equation}\label{mvaralpha}
\begin{cases} (i) &  -\Delta\bar u+ \nabla\bar u\cdot (\nabla u+P)  -V_m(y,\alpha m)\alpha\bar m -V_m(y,\alpha m) m =k(P,\alpha)\\ (ii)&
-\Delta \bar m-{\rm div}\big((P+\nabla u)\bar m \big) ={\rm div }(m\nabla \bar u) \\
& \int_{\mathbb{T}^n}\bar m=\int_{\mathbb{T}^n}\bar u=0. \end{cases}
\end{equation} 
Moreover the solution is unique and smooth, i.e. 
$(\bar u, \bar m)	\in C^{2,\gamma}\times  W^{1,p}$ for all $\gamma \in (0,1)$ and all $p>1$.

Moreover \[k(P,\alpha)=   -\int_{\mathbb{T}^n}\left[ V_m(y,\alpha m) (m+\alpha\bar m)^2 +\alpha m|\nabla \bar u|^2\right] dy.\]
\end{lemma} 
\begin{proof} The proof is completely analogous (with some minor modifications) to the proof of Lemma \ref{lemmavar}, so we omit it.

To obtain the representation formula for $k(P,\alpha)$, we multiply equation (i) in  \eqref{mvaralpha} by $(m+\alpha\bar m)$, where $m$ solves (ii) in \eqref{ucor}, and integrate on $\mathbb{T}^n$. We obtain, after some integration by parts, 
\[k(P, \alpha)= \int_{\mathbb{T}^n} \alpha\ {\rm div }(m \nabla \bar u )\bar u -V_m(y,\alpha m)(m+\alpha\bar m)^2 dy\] which gives the desired conclusion. \end{proof}
\begin{theorem} \label{estideltaalpha} Let $\alpha>0$. 
  \[\lim_{\delta\to 0}\frac{\bar H(P, \alpha+\delta)-\bar H(P, \alpha)}{\delta}=  -\int_{\mathbb{T}^n}\left[ V_m(y,\alpha m)  (m+\alpha\bar m)^2 +\alpha m|\nabla \bar u|^2\right] dy<0 \]
 where $m$ is the solution to (ii) in \eqref{ucor} and $(\bar u, \bar m)$ is the unique solution to \eqref{mvaralpha}.
In particular   $\bar H$ is strictly decreasing in $\alpha$. 
\end{theorem} 
\begin{proof}The proof is completely analogous (with some minor modifications) to the proof of Theorem \ref{hbar}, so we omit it.
\end{proof} 
\subsection{Regularity properties of $\bar b$}

\begin{theorem} \label{bbaresti} Let $P\in\mathbb{R}^n$ and  $\alpha>0$. 
Then $\bar b$ is locally Lipschitz continuous with respect to $P$ and $\alpha$. Moreover 
\[\lim_{\delta\to 0}\frac{\bar b(P+\delta e_i, \alpha)-\bar b(P,\alpha)}{\delta}= e_i+ \int_{\mathbb{T}^n}\left[\tilde m_i \nabla u +m \nabla \tilde u_i\right]dy\]
where $(\tilde u_i, \tilde m_i)$ is the solution to \eqref{mvarP} and
\[\lim_{\delta\to 0}\frac{\bar b(P, \alpha+\delta)-\bar b(P,\alpha)}{\delta}=  \int_{\mathbb{T}^n}\left[\bar m \nabla u  +m \nabla \bar u\right]dy\]where $(\bar u, \bar m)$ is the solution to \eqref{mvaralpha}.
\end{theorem} \begin{proof}
Note that \[\frac{\bar b(P+\delta e_i, \alpha)-\bar b(P,\alpha)}{\delta}=  \int_{\mathbb{T}^n}\left[n^\delta_i(\nabla u_i^\delta+P+\delta e_i)+m (\nabla w^\delta_i+e_i)\right]dy \] where $n^\delta_i$ and $w^\delta_i$ are defined in \eqref{deltaw}. Note that the r.h.s. in the previous equality    is  bounded by a constant depending on $P, \alpha, u,m$  due to Proposition \ref{estidelta} and to \eqref{limdelta}. This gives locally Lipschitz continuity of $\bar b$ with respect to $P$. 

So, by the proof of Theorem \ref{hbar} and by \eqref{limdelta},
 \[\lim_{\delta\to 0}\frac{\bar b(P+\delta e_i, \alpha)-\bar b(P,\alpha)}{\delta}=  \int_{\mathbb{T}^n}\left[\tilde m_i(\nabla u +P )+m (\nabla \tilde u_i+e_i)\right]dy.\]

Analogous argument gives the statement for the variation with respect to $\alpha$. \end{proof}

\section{Qualitative properties of the effective operators}\label{sect:qual}
In this section we provide some qualitative properties of the effective operators $\bar H, \bar b$: their asymptotic limit as $|P|\to +\infty$ and an explicit example where the effective system~\eqref{systemeffintro} loses the  MFG structure, 
i.e. $\nabla_P \bar H\neq \bar b$.

\begin{proposition}\label{proplimit}
Let $\bar H$ and $\bar b$ be the effective operators defined in \eqref{ucor} and \eqref{bbar}. Then
\[\lim_{|P|\to +\infty}\frac{\bar H(P, \alpha)}{|P|^2} =\frac{1}{2} \ \qquad\text{and}\qquad \lim_{|P|\to +\infty}\frac{|\bar b(P, \alpha)-P|}{|P|}=0, \]
uniformly for $\alpha \in [0,+\infty)$.
\end{proposition}
\begin{proof}
We multiply   equation (i) in \eqref{ucor} by $\frac{m-1}{|P|^2}$. We integrate and we get, recalling the periodicity assumptions  and that $m$ has mean $1$, 
\begin{eqnarray}\label{c1dim1}0=\int_{\mathbb{T}^n} -\frac{m\Delta u }{|P|^2}+\frac{1}{2|P|^2}\left|\nabla u +P\right|^2  (m-1)-\frac{V(y,\alpha m)}{|P|^2} (m-1) =\\ =\int_{\mathbb{T}^n}  -\frac{m \Delta u}{|P|^2}+\frac{|\nabla u |^2}{2|P|^2} (m-1)+ \nabla u\cdot \frac{P}{|P|^2} m-\frac{V(y,\alpha m)}{|P|^2} (m-1). \nonumber \end{eqnarray}
We multiply the second equation in \eqref{ucor} by $\frac{u}{|P|^2}$, integrate and we get, recalling the periodicity assumptions,
\begin{equation}\label{c2dim1}0=\int_{\mathbb{T}^n} -\frac{ u \Delta m }{|P|^2} + m \frac{|\nabla u|^2}{|P|^2}+m \nabla u \cdot \frac{P}{|P|^2}.  \end{equation}
We subtract \eqref{c2dim1} by \eqref{c1dim1} and get
\begin{equation}\label{c3dim1} \int_{\mathbb{T}^n}  \frac{|\nabla u|^2}{2|P|^2} ( m+1) +\frac{V(y,\alpha m)}{|P|^2}  (m-1)=  0. \end{equation}  
We observe that, by monotonicity of $V$, \[V(y,\alpha m)(m -1) \geq  V(y,\alpha)(m  -1).\]

So \eqref{c3dim1} gives
\begin{equation*}\label{c4dim1} \int_{\mathbb{T}^n}  \frac{|\nabla u |^2}{2|P|^2}(m+1) +\frac{V(y,\alpha )}{|P|^2}  m dy \leq   \int_{\mathbb{T}^n}\frac{V(y,\alpha )}{|P|^2} dy. \end{equation*}
Recalling that $V\geq 0$ and $m\geq 0$, we obtain \begin{equation}\label{c8dim1} \int_{\mathbb{T}^n} \frac{|\nabla u |^2}{2|P|^2}(m+1) dy \leq  \frac{ \int_{\mathbb{T}^n} V(y,\alpha )dy}{|P|^2}.\end{equation} 
Since $m\geq 0$,  by Proposition \ref{cellprop},   this implies that $\frac{\nabla u}{|P|}\to 0$ in $L^2$ as $|P|\to +\infty$ (and then also 
$\frac{  u}{|P|}\to 0$ in $L^2 $ as $|P|\to +\infty$ by the Poincar\'e  inequality).  

We multiply by $\frac{1}{|P|^2}$ the first equation in \eqref{ucor} and integrate to get 
\begin{equation}\label{c1} \int_{\mathbb{T}^n}  \frac{|\nabla u+P|^2}{2|P|^2}-\frac{V(y,\alpha m)}{|P|^2}dy=\frac{\bar H(P,\alpha)}{|P|^2}.\end{equation}
We recall that $V$ is bounded and moreover by the triangle inequality 
\[\frac{1}{\sqrt{2}}-\left(\int_{\mathbb{T}^n}  \frac{|\nabla u |^2}{ 2|P|^2}\right)^{1/2}\leq \left(\int_{\mathbb{T}^n}  \frac{|\nabla u+P|^2}{2|P|^2}dy\right)^{1/2}\leq \frac{1}{\sqrt{2}}+\left(\int_{\mathbb{T}^n}  \frac{|\nabla u |^2}{2 |P|^2}\right)^{1/2}.\] Then  \[\lim_{|P|\to +\infty}\left(\int_{\mathbb{T}^n}  \frac{|\nabla u+P|^2}{2|P|^2}dy\right)^{1/2}= \frac{1}{\sqrt{2}}.\]
So, letting  $|P|\to +\infty$ in \eqref{c1}, we get the desired result. 

By  definition \eqref{bbar}, we get \[\bar b(P, \alpha)-P=\int_{\mathbb{T}^n} m(y) \nabla u(y) dy.\] 
So, \[\frac{\left|\bar b(P, \alpha)-P\right|}{|P|} \leq \frac{1}{|P|}\int_{\mathbb{T}^n} m(y) |\nabla u(y)| dy.\] 
By H\"older's  inequality and \eqref{c8dim1} we get
\[\frac{\left|\bar b(P, \alpha)-P\right|}{|P|} \leq  \left(\int_{\mathbb{T}^n}  \frac{|\nabla u |^2}{  |P|^2} m dy \right)^{\frac{1}{2}} \left(\int_{\mathbb{T}^n}   m dy \right)^{\frac{1}{2}}\leq  \frac{ \sqrt{2\int_{\mathbb{T}^n} V(y,\alpha )dy}}{|P|}.\]
So, letting $|P|\to +\infty$, we conclude that  \[\frac{\left|\bar b(P, \alpha)-P\right|}{|P|}\to 0.\]
\end{proof}
%

\subsection{A case in which the limit system is not MFG}
We show that in general  the effective  system \eqref{systemeffintro} is not a MFG system by providing an example in dimension $n=1$  where  \begin{equation}\label{exunodim}\nabla_P \bar H(P,\alpha)\neq  \bar b (P,\alpha).\end{equation}

We assume that the potential has the following form  \[V(y,m)=v(y)+m,\textrm{with }v\geq 0.\] So the cell problem reads
\begin{equation}\label{celldim1} \begin{cases} 
 - \Delta u     \!+\!\frac{1}{2}|\nabla u  \!+\!P|^2\!- \alpha m -v(y) =
 \bar  H(P,\alpha), & y\in \mathbb{T}^n\\
 -\Delta m  - {\rm div}\left(m \left(\nabla u +P\right)\right)=0,& y\in\mathbb{T}^n \\   \int_{\mathbb{T}^n} u(y)dy =0 \qquad  \int_{\mathbb{T}^n} m(y)dy=1. & \end{cases} 
\end{equation}
Note that in this case the potential does not satisfy the standing assumptions,
 since it is not bounded. Moreover, we will see that the previous results still apply.
 
\begin{lemma}\label{lemmaunodim} Let $n\leq 3$. 

\begin{itemize}
\item[(i)] For every $P\in\mathbb{R}^n$, $\alpha\geq 0$ there exists a unique constant $\bar H(P,\alpha)$ such that
\eqref{celldim1} admits a solution $(u,m)$. Moreover this solution is unique, $m>0$ and 
$u\in C^{2,\gamma}$, $m\in W^{1,p}$ for every $\gamma\in (0,1)$ and $p>1$.
\item[(ii)] There hold $\lim_{|P|\to +\infty}\frac{\bar H(P, \alpha)}{|P|^2} =\frac{1}{2}$ and $\lim_{|P|\to +\infty}\frac{|\bar b(P, \alpha)-P|}{|P|}=0,$
locally uniformly for $\alpha \in [0,+\infty)$.
\item[(iii)] The maps  $(P,\alpha)\to \bar H(P, \alpha), \bar b(P,\alpha)$ are continuous. 
\item[(iv)]  There holds \[\nabla_P \bar H(P,\alpha)=     \bar b(P,\alpha)-\frac{\alpha}{2}\nabla_P(\|m\|_{L^2}^2).\] \end{itemize}
\end{lemma} 
\begin{proof}
(i) This existence result can be found in \cite[Thm 1.4]{cirant}, see also \cite{pv} and \cite{gps2}. 

%
%

(ii) Arguing as in the proof of Proposition \ref{proplimit}, we obtain \eqref{c3dim1}, which in this case reads
\begin{equation}\label{c5dim1}\int_{\mathbb{T}^n}   \frac{|\nabla u|^2}{2|P|^2} ( m+1) +\frac{v(y)+\alpha m}{|P|^2}  (m-1) dy=  0. \end{equation}
This implies, recalling that $v\geq 0$ 
\begin{equation}\label{c9dim1}\int_{\mathbb{T}^n}   (m+1)\frac{|\nabla u|^2}{2|P|^2}  +\frac{ \alpha m^2}{|P|^2}  \leq   \frac{\int_{\mathbb{T}^n} v(y)dy+\alpha }{|P|^2} \end{equation}
which in turns gives, since $m\geq 0$,  that   $\nabla u/|P|, u/|P| \to 0$ in $L^2$ as $|P|\to +\infty$, and that, locally uniformly in $\alpha\geq 0$, \[\frac{\bar H(P,\alpha)}{|P|^2}= \int_{\mathbb{T}^n} \frac{|\nabla u+P|^2}{2|P|^2}-\frac{v(y)+\alpha m}{|P|^2}dy\to \frac{1}{2} \qquad \textrm{as }|P|\to+\infty.\]
Moreover, arguing as in the proof of Proposition \ref{proplimit}, we get also 
\[\frac{|\bar b(P,\alpha)-P|}{|P| }\to 0 \qquad \textrm{as }|P|\to+\infty, \quad \textrm{loc. unif. in }\alpha\geq 0.\]

(iii) Integrating the first equation in \eqref{celldim1}  we get 
\[\bar H(P,\alpha)\geq \frac{|P|^2}{2} -\alpha-\int v(y)dy.\] Moreover, if we multiply the first equation in \eqref{celldim1} by $m$, the second by $u$ and integrate, we get,  that 
\[\bar H(P,\alpha)\leq \frac{|P|^2}{2}.\] 


Arguing again as in the proof of Proposition \ref{4.2}, we get the statement.

(iv) By (iii), the same arguments in the proof of Theorem \ref{hbar} apply. So,  using the explicit formula of $V$, we have that \[\frac{d}{d P_i} \bar H(P,\alpha)=   \bar b_i(P,\alpha)-\int_{\mathbb{T}^n} \alpha \tilde m_i m dy=   \bar b_i(P,\alpha)-\frac{\alpha}{2}\frac{d}{d P_i}(\|m\|_{L^2}^2) \] where $\tilde m_i$ is the solution to \eqref{mvarP}. 
\end{proof}
We observe that, for every $P$, $m$ cannot be a constant; actually, it would imply that also $u$ is constant which would contradict the first equation in~\eqref{celldim1}.   Jensen's inequality implies
\begin{equation}\label{jensen}
\|m\|_{L^2}>1=\|m\|_{L^1}^2.
\end{equation}
Assume for the moment that up to a subsequence 
\begin{equation}\label{cv2}\|m\|_{L^2}\to 1\qquad \textrm{as }|P|\to+\infty,\end{equation} then necessarily $\frac{d}{d P_i}(\|m\|_{L^2}^2)\neq 0$ at some  $i, P,\alpha$ and so \eqref{exunodim} holds true. Then we are left with the proof of~\eqref{cv2}.

\begin{proposition}\label{6.2} Let  $n=1$ and $(u,m)$ be the solution to \eqref{celldim1}. 
Then \eqref{cv2} holds true. 
\end{proposition}
\begin{proof}
By \eqref{c9dim1} we get   that $\|m\|_{L^2}$  and $\|u\|_{W^{1,2}}$  are uniformly bounded with respect to $|P|$. 

Now we prove that $\sqrt{m}$ is uniformly bounded with respect to $P$ in $W^{1,2}$ using some argument of \cite[Lemma 2.5]{cllp1}.
We multiply the second equation  in \eqref{celldim1} by $\log m$ and  integrate.
We get, using periodicity, 
\[\int_{\mathbb{T}^1} \frac{| m'|^2}{m} + m' u' dy=0.\]So, by Cauchy-Schwarz inequality
\[\int_{\mathbb{T}^1}\frac{|m'|^2}{m}dy=-\int_{\mathbb{T}^1} m' u' dy\leq \frac{1}{2}\int_{\mathbb{T}^1}\frac{|m'|^2}{m}dy+ \frac{1}{2}\int_{\mathbb{T}^1} m |u'|^2 dy.\]
Therefore, by \eqref{c9dim1}, we conclude  \[\int_{\mathbb{T}^1}  |(\sqrt{m})'|^2 dy=\frac{1}{4}\int_{\mathbb{T}^1} \frac{|m'|^2}{m}dy\leq\frac{1}{4}\int_{\mathbb{T}^1} m |u'|^2 dy\leq \frac{\alpha+\int_{\mathbb{T}^1} v(y)dy}{4}.\]
Since $n=1$,   the embedding of $W^{1,2}(0,1)$ in $C(0,1)$ is compact, so we  conclude that, possibly passing to a subsequence, $\sqrt{m}\to \sqrt{m_\infty}$ uniformly as $|P|\to +\infty$. This implies that $m\to m_\infty$ in $C(0,1)$ as $P\to +\infty$, then also strongly in $L^2$. 

We are left to prove that $m_\infty=1$. 

It is sufficient to prove that $m\to 1$ weakly in $L^2$ as $|P|\to +\infty$. 
We apply  to the second equation in \eqref{celldim1}  a smooth periodic test function $\phi$ and we divide by $|P|$. By periodicity, we get
\[\int_{\mathbb{T}^1} m \phi' dy=\int_{\mathbb{T}^1} \frac{m  \phi''}{|P|} - \frac{m \phi' \cdot u'}{|P|} dy.
\]
Letting $|P|\to +\infty$ and recalling that $\nabla u/|P|\to 0$ in $L^2$ and that $\|m \|_{L^2}$ is uniformly bounded with respect to $|P|$, we get
\begin{equation}\label{pphi} 
\lim_{|P|\to+\infty}\int_{\mathbb{T}^1} m   \phi' dy=0, 
\end{equation}
for every smooth periodic  test function $\phi$. We observe that every smooth periodic function
$\psi$ can be written as $c+\phi'$, where $c=\int_0^1 \psi$ and $\phi$ is still periodic.
So, since $\int_0^1m=1$ we get  
\begin{equation*}\lim_{|P|\to+\infty}\int_{\mathbb{T}^1} (m-1)\psi dy = \lim_{|P|\to+\infty}\int_{\mathbb{T}^1} [(m-1) c  + m\phi'-\phi']dy =  \lim_{|P|\to+\infty}\int_{\mathbb{T}^1}  m\phi'dy=0.
\end{equation*}

%
%
%

\end{proof}
%
%

\section{Convergence for affine-constant initial data}\label{sect:affine}

We prove the homogenization result in the special case where  initial and terminal data of the system \eqref{systemintro}
are affine and constant, respectively. Our arguments are based on some a priori estimates which are inspired by estimates used in \cite{cllp1} for 
investigating the  long time behaviour for MFG systems.

Fix $P\in\mathbb{R}^n$ and consider the mean field game system
\begin{equation}\label{system1p}\begin{cases}
-u_t^\epsilon-\epsilon\Delta u^\epsilon +\frac{1}{2}|\nabla u^\epsilon|^2=V\left(\frac{x}{\epsilon}, m^\epsilon \right), & x\in \mathbb{R}^n, t\in (0,T)\\
m^\epsilon_t-\epsilon \Delta m^\epsilon-{\rm div}(m^\epsilon\nabla u^\epsilon)=0,& x\in \mathbb{R}^n, t\in (0,T)
\\ u^\epsilon(x,T)=P\cdot x\qquad m^\epsilon(x, 0)\equiv 1& x\in \mathbb{R}^n.
\end{cases}\end{equation}
Actually the initial data $P\cdot x$ is not periodic, so Proposition \ref{ex} does not apply directly.
In order to preserve the periodicity, for now on $\epsilon=\frac{1}{k}$, with $k\in \mathbb{N}$.  

We construct a solution as follows.

\begin{lemma}\label{lemmaper} There exists a  smooth solution $(u^\epsilon, m^\epsilon)$ to the mean field game system
\eqref{system1p}. Moreover  the maps $x\mapsto u^\epsilon(x,t)-P\cdot x$, $x\mapsto \nabla u^\epsilon(x,t)$, $x\mapsto m^\epsilon(x,t)$ are $\epsilon\mathbb{Z}^n$ periodic for every $t\in [0,T]$. Finally the solution is unique among all solutions such that $u^\epsilon(x,t)-P\cdot x$ and $m^\epsilon$ are $\epsilon\mathbb{Z}^n$ periodic.  \end{lemma}
\begin{proof}
Let $(w^\epsilon, m^\epsilon)$ be the unique $\epsilon\mathbb{Z}^n$ periodic smooth  solution to  the mean field game system
\begin{equation}\label{systemp}\begin{cases}
-w_t^\epsilon-\epsilon\Delta w^\epsilon +H^P(\nabla w^\epsilon)=V\left(\frac{x}{\epsilon}, m^\epsilon(x, t)\right), & x\in \mathbb{R}^n, t\in (0,T)\\
m^\epsilon_t-\epsilon \Delta m^\epsilon-{\rm div}(m^\epsilon \nabla{_P}  H^P(\nabla w^\epsilon))=0,& x\in \mathbb{R}^n, t\in (0,T)\\ w^\epsilon(x,T)=0\qquad m^\epsilon(x,0)=1
\end{cases}\end{equation} where \[H^P(q)=\frac{|q+P|^2}{2}.\] Note that $H^P$ is strictly convex and has superlinear growth,
moreover $V$ is monotone in the second argument, so \eqref{systemp} admits a unique $\epsilon\mathbb{Z}^n$ periodic solution (see \cite[Lemma 4.2]{cllp1}).

Define $u^\epsilon(x,t)= w^\epsilon(x,t)+P\cdot x$. Then $(u^\epsilon, m^\epsilon)$ is a smooth solution to \eqref{system1p}.  Finally $u^\epsilon(x,t)-P\cdot x$, $m^\epsilon(x,t)$ and  $\nabla u^\epsilon$ are $\epsilon\mathbb{Z}^n$ periodic with respect to $x$. 
\end{proof}

We recall the definition of $\bar H(P,\alpha)$ and $\bar b (P,\alpha)$ given respectively in Proposition \ref{cellprop} and in equation \eqref{bbar}. In particular $\bar H(P,1)$ is the  unique constant such that there exists a  solution $(u^1,m^1)$ to the cell system
  \begin{equation}\label{cell1} \begin{cases}-\Delta u^1 \!+\!\frac{1}{2}|\nabla u^1 \!+\!P|^2\!-\!V(y,   m^1 ) =
 \bar H(P,1)\\  -\Delta m^1 -{\rm div}\left(m^1 \left(\nabla u^1 +P\right)\right)=0,\\   \int_{\mathbb{T}^n} u^1=0 \qquad \int_{\mathbb{T}^n} m^1=1
\end{cases}  \end{equation} and  \[\bar b(P,1)=\int_{\mathbb{T}^n} (\nabla u^1+P)m^1(y)dy.\]
We recall that, by \cite[Thm 4.3, pag 136]{bbook}, $m^1>0$; in particular,  there exist constants $\delta>0$ and $K>0$ such that $0<\delta\leq m^1(y)\leq K$.

It is easy to check that $(u_0,m_0)$ with
\begin{equation}\label{u0m0} u^0(t,x)= P\cdot x +(t-T) \bar H(P,1)\qquad m^0(t,x)\equiv 1\end{equation} is a solution to the limit system
\begin{equation}\label{limitp}\begin{cases} -u_t+\bar H(\nabla u, m)=0&\qquad x\in\mathbb{R}^n,\, t\in(0,T)\\ m_t-{\rm  div }(\bar b(\nabla u, m) m)=0&\qquad x\in\mathbb{R}^n,\, t\in(0,T)\\ u(x,T)=P\cdot x, \ \ m(x,0)=1&\qquad x\in\mathbb{R}^n. \end{cases}\end{equation}
Let us now state our convergence result whose proof is postponed to section \ref{proofconv}.

\begin{theorem}\label{thm:conv}
Let $(u^\epsilon, m^\epsilon)$ be the solution of the mean field game system \eqref{system1p} defined in Lemma \ref{lemmaper}. Then for every compact $Q$ in $\mathbb{R}^n$, 
\begin{itemize} \item[i)] $u^\epsilon\to u^0$ in $L^2(Q\times[0,T])$ \item[ii)] $m^\epsilon\to m^0$ weakly in $L^1(Q\times[0,T])$ \end{itemize}
where $(u^0,m^0)$ is the solution to the effective problem \eqref{limitp} defined in \eqref{u0m0}.
\end{theorem}

\subsection{A priori estimates}

Let $(u^\epsilon, m^\epsilon)$ the solution to \eqref{system1p} given by Lemma \ref{lemmaper}. Consider the functions 
\[\begin{cases} v^\epsilon(y,t)=\frac{1}{\epsilon}u^\epsilon( \epsilon y, t)-   P\cdot y-\frac{1}{\epsilon} (t-T)\bar H(P, 1)-  u^1\left(y\right)\\
n^\epsilon( y,t)=m^\epsilon( \epsilon y, t)-m^1\left( y\right). \end{cases}\]

By Lemma \ref{lemmaper},  we get that $v^\epsilon, n^\epsilon$ are both $ \mathbb{Z}^n$-periodic and smooth functions.  It is a straightforward computation to check that they solve
\begin{equation} \label{s1} \begin{cases} -\epsilon v^\epsilon_t-  \Delta  v^\epsilon+H_P(y, \nabla v^\epsilon)= V\left(y, m^1+n^\epsilon\right)- V\left(y, m^1 \right)&\quad y\in\mathbb{T}^n,\,t\in(0,T)\\  
\epsilon n^\epsilon_t- \Delta n^\epsilon-{\rm div}(n^\epsilon \nabla_q H_P(y, \nabla v^\epsilon))={\rm div} (m^1 \nabla v^\epsilon) &\quad y\in\mathbb{T}^n,\,t\in(0,T)\\
v^\epsilon(y,T)=- u^1\left(y\right)\quad n^\epsilon(  y,0)=1-m^1\left(y\right)&\quad y\in\mathbb{T}^n  \end{cases}\end{equation}

where \begin{equation}\label{hp} H_P(y,q)=\frac{|q|^2}{2}+q\cdot(P+\nabla u^1(y)). \end{equation}

Note that system \eqref{s1} is not a mean field game system because of the presence of the term ${\rm div} (m^1 \nabla v^\epsilon)$. Nevertheless,  due to the divergence structure,  since $\int_{\mathbb{T}^n} n^\epsilon(y,0)dy=0$ and $n^\epsilon( \cdot,t)$, $\nabla v^\epsilon( \cdot,t)$,
$\nabla u^1(\cdot)$ are $\mathbb{Z}^n$-periodic,   we get that for every $t\in [0,T]$, \begin{equation}\label{nmediazero}\int_{\mathbb{T}^n} n^\epsilon ( y,t)dy=0.\end{equation}

\begin{lemma} \label{estimate1} For every $t_1, t_2\in [0,T]$ with $t_1\leq t_2$, we get
\begin{equation} \label{est1}
-\epsilon\left[ \int_{\mathbb{T}^n}  v^\epsilon n^\epsilon \right]_{t_1}^{t_2}= \int_{t_1}^{t_2}\int_{\mathbb{T}^n} \frac{2m^1+n^\epsilon}{2} |\nabla v^\epsilon|^2 + \left[V(y, m^1+n^\epsilon)-V(y,m^1)\right]n^\epsilon dy.\end{equation} \end{lemma}
\begin{proof}
We multiply the first equation in \eqref{s1} by $n^\epsilon$ and the second by $v^\epsilon$  subtract the first equation from the second,  and integrate in $\mathbb{T}^n\times [t_1,t_2]$:
\begin{eqnarray*} 0 &= &\int_{t_1}^{t_2} \int_{\mathbb{T}^n}-\epsilon( v^\epsilon n^\epsilon)_t-\Delta v^\epsilon n^\epsilon+\Delta n^\epsilon v^\epsilon+{\rm div} (n^\epsilon (\nabla v^\epsilon+P+\nabla u^1))v^\epsilon    \\& & +\int_{t_1}^{t_2} \int_{\mathbb{T}^n}n^\epsilon\left[\frac{|\nabla v^\epsilon|^2}{2}-\left(V(y, m^1+n^\epsilon)-V(y,m^1)\right) +\nabla v^\epsilon\cdot (P+\nabla u^1)\right]\\ && + \int_{t_1}^{t_2} \int_{\mathbb{T}^n}{\rm div} (m^1 \nabla v^\epsilon)v^\epsilon.\end{eqnarray*} 
Recalling the periodicity of $n^\epsilon, v^\epsilon, m^1$, we get  \eqref{est1}. \end{proof}

We define the following functional on $H^1(\mathbb{T}^n)\times H^1(\mathbb{T}^n) $
\begin{equation}\label{en} E(v,n)=\int_{\mathbb{T}^n} (n(y)+m^1(y)) H_P(y, \nabla v(y))+ \nabla v(y)\cdot (\nabla n(y)+\nabla m^1(y))-\Phi^{ 1}(y,n)dy
\end{equation}
where $(u^1, m^1)$ is the solution to the cell problem \eqref{cell1}, $H_P$ has been defined in \eqref{hp}, \begin{equation}\label{phi} \Phi^1(y,n)=\int_0^n V(y, s+m^1)-V(y,m^1)ds.\end{equation}  Note that  due to the fact that $V$ is increasing in the second variable, $\Phi^1\geq 0$.
\begin{lemma}\label{lemmaenergy}
Let $(v^\epsilon,n^\epsilon)$ the solution to \eqref{s1}. Then
  there exists a constant $C_P $  depending on $P$ (independent of $\epsilon$) such that  \[ E(v^\epsilon( \cdot, t), n^\epsilon(\cdot,t))\leq C_P\qquad \forall t\in [0,T].\]
Moreover   \begin{equation}\label{stimal2} \|\nabla v^\epsilon( \cdot,0)\|_{L^2(\mathbb{T}^n)}\leq C_P.\end{equation}
\end{lemma}
\begin{proof} First of all we compute
\begin{multline*}
\frac{d}{dt} E(v^\epsilon( \cdot,t), n^\epsilon( \cdot,t))= \int_{\mathbb{T}^n}  n^\epsilon_t   H_P(y, \nabla v^\epsilon)+ (n^\epsilon  +m^1 ) \nabla_q H_P(y, \nabla v^\epsilon ) \nabla v^\epsilon_t dy+\\+\int_{\mathbb{T}^n}   \nabla v^\epsilon_t \cdot (\nabla n^\epsilon +\nabla m^1 )+ \nabla v^\epsilon  \cdot  \nabla n^\epsilon_t  -\Phi^{ 1}_n(y,n^\epsilon) n^\epsilon_t dy;\end{multline*} 
integrating by parts and recalling \eqref{s1} and \eqref{cell1}, we get    
\begin{eqnarray*}\frac{d}{dt} E(v^\epsilon( \cdot,t), n^\epsilon( \cdot,t))&=& \int_{\mathbb{T}^n}  n^\epsilon_t \left[H_P(y, \nabla v^\epsilon( y,t))-\Delta v^\epsilon -V(y, n^\epsilon+m^1)+V(y, m^1)\right]\\&&+
\int_{\mathbb{T}^n}  v^\epsilon_t\left[-{\rm div}( n^\epsilon \nabla_q H_P(y, \nabla v^\epsilon)-{\rm div}( m^1 \nabla v^\epsilon )-\Delta n^\epsilon  \right] \\&&+\int_{\mathbb{T}^n}  v^\epsilon_t\left[-\Delta m^1-{\rm div} (m^1(P+\nabla u^1))\right] =\int_{\mathbb{T}^n} \epsilon  v^\epsilon_t n^\epsilon_t-\epsilon v^\epsilon_t n^\epsilon_t\\&=&0. \end{eqnarray*}
Therefore for every $t\in[0,T]$
\[ E(v^\epsilon( \cdot,t), n^\epsilon( \cdot,t))\equiv M_T\qquad \forall t\in[0, T].\] 
Then, recalling that $\Phi^1\geq 0$ and  $v^\epsilon(y,T)=-u^1(y)$ \[M_T= E(v^\epsilon( \cdot,T), n^\epsilon( \cdot,T))\leq  \int_{\mathbb{T}^n} (\Delta u^1+H_P(y, -\nabla u^1)) (n^\epsilon( y,T)+m^1(y)).\] By definition  of $H_P$ in \eqref{hp} and by the fact that $n^\epsilon+m^1\geq 0$ with  $\int_{\mathbb{T}^n} n^\epsilon(y,T)+m^1(y)=1$ (by \eqref{nmediazero}) 
\begin{eqnarray*} M_T&\leq&  \int_{\mathbb{T}^n}
\left(\Delta u^1-\nabla u^1\cdot P-\frac{|\nabla u^1|^2}{2}\right)(n^\epsilon( y,T)+m^1(y))\\&\leq&  \int_{\mathbb{T}^n} \left(\Delta u^1+\frac{|P|^2}{2}\right)(n^\epsilon( y,T)+m^1(y))dy\leq \|u^1\|_{C^2}+\frac{|P|^2}{2}:=C_P\end{eqnarray*}
which amounts to the first part of the statement. Let us now pass to prove \eqref{stimal2}. Recalling that $n^\epsilon(0, y)+m^1(y)\equiv 1$,
\begin{eqnarray*} C_P&\geq&  E(v^\epsilon(0,\cdot), n^\epsilon(0,\cdot))=\int_{\mathbb{T}^n} H_P(y, \nabla v^\epsilon(0,y) )- \Phi^1(y, 1-m^1(y))\\ &\geq &\int_{\mathbb{T}^n} \frac{1}{4} |\nabla v^\epsilon(0,y)|^2 - |P+\nabla u^1|^2- \Phi^1(y, 1-m^1(y)) dy\end{eqnarray*}
where the last inequality is a consequence of the Cauchy-Schwartz inequality.
Note that by definition \eqref{phi} we have that \[ \Phi^1(y, 1-m^1(y))=\int_0^{1-m^1(y)}  V(y,s+m^1)- V(y, m^1) ds\leq (V(y,1)-V(y, m^1))(1-m^1) . \]
Then,
\[\int_{\mathbb{T}^n} |\nabla v^\epsilon(0,y)|^2\leq 4 C_P+ 4\int_{T^n} \left[|P+\nabla u^1|^2+(V(y,1)-V(y, m^1))(1-m^1)\right]dy, \]
 which  implies that
$|\nabla v^\epsilon(0, \cdot)|$ is bounded in $L^2(\mathbb{T}^n)$.
\end{proof}
We are ready to prove the main lemma on a priori bounds.
\begin{lemma}\label{lemmaconv} There exists a constant $C_P$ depending on $P$ such that
\[\int_{0}^{T}\int_{\mathbb{T}^n} \frac{2m^1+n^\epsilon}{2} |\nabla v^\epsilon|^2 + \left[V(y, m^1+n^\epsilon)-V(y,m^1)\right]n^\epsilon dy\leq C_P\epsilon.\]
Moreover \[\lim_{\epsilon\to 0}\int_{0}^{T}\int_{\mathbb{T}^n}   |\nabla v^\epsilon(y,t)|^2dydt =0\qquad \textrm{and }\qquad  \lim_{\epsilon\to 0} \int_{0}^{T}\int_{\mathbb{T}^n}   |n^\epsilon(y,t)|dydt= 0. \]
\end{lemma}
\begin{proof}
By the initial/terminal data of system \eqref{s1} and since $\int_{\mathbb{T}^n}m^1=1$, we rewrite \begin{eqnarray*}  \left[ \int_{\mathbb{T}^n}  v^\epsilon n^\epsilon \right]_{0}^{T} &=&   \int_{\mathbb{T}^n} v^\epsilon(0,y)(m^1(y)-1)dy- \int_{\mathbb{T}^n} n^\epsilon(T,y)u^1(y)\\ &= &   \int_{\mathbb{T}^n} (v^\epsilon(0,y)-\int_{\mathbb{T}^n}
v^\epsilon(0,z)dz)(m^1(y)-1)dy\\\ && -  \int_{\mathbb{T}^n} (n^\epsilon(T,y)+m^1(y))u^1(y)+  \int_{\mathbb{T}^n} m^1(y)u^1(y). \end{eqnarray*} So using Cauchy-Schwarz and  Poincar\'e  inequalities in the first term, we get
\begin{eqnarray*} \left| \left[ \int_{\mathbb{T}^n}  v^\epsilon n^\epsilon \right]_{0}^{T} \right| &\leq &   \left( \int_{\mathbb{T}^n} \left(v^\epsilon(0,y)-\int_{\mathbb{T}^n}
v^\epsilon(0,z)dz\right)^2 dy \right)^{1/2} \left( \int_{\mathbb{T}^n}((m^1(y))^2-1)dy\right)^{1/2}\\\ && + \int_{\mathbb{T}^n} (n^\epsilon(T,y)+m^1(y))|u^1(y)| dy+  \int_{\mathbb{T}^n} m^1(y)|u^1(y)|dy\\  & \leq & C \|\nabla v^\epsilon(0,\cdot)\|_{L^2(\mathbb{T}^n)}+ \|u^1\|_\infty   \int_{\mathbb{T}^n} (n^\epsilon(T,y)+m^1(y))dy+ \|u^1\|_\infty\|m^1\|_\infty \\ &=& C \|\nabla v^\epsilon(0,\cdot)\|_{L^2(\mathbb{T}^n)}+  K_P, \end{eqnarray*}  where $K_P=\|u^1\|_\infty(1+\|m^1\|_\infty)$.
So, by Lemma \ref{estimate1} and by \eqref{stimal2}, we get that
\begin{multline}\label{estimateint}\int_{0}^{T}\int_{\mathbb{T}^n} \frac{2m^1+n^\epsilon}{2} |\nabla v^\epsilon|^2 + \left[V(y, m^1+n^\epsilon)-V(y,m^1)\right]n^\epsilon dy \\ =-\epsilon\left[ \int_{\mathbb{T}^n}  v^\epsilon n^\epsilon \right]_{0}^{T}\leq C_P\epsilon \end{multline} and in particular, 
by the monotonicity of $V$,
\[\int_{0}^{T}\int_{\mathbb{T}^n} \frac{2m^1+n^\epsilon}{2} |\nabla v^\epsilon|^2 \leq C_P\epsilon. \] Now we recall that $m^\epsilon(t,\epsilon y)=m^1(y)+n^\epsilon(t,y)\geq 0$ for every $t\in [0,T]$ and $y\in \mathbb{T}^n$, moreover $m^1\geq \delta>0$. So we get that \[\int_{0}^{T}\int_{\mathbb{T}^n}   |\nabla v^\epsilon|^2 \leq \frac{C_P}{\delta}\epsilon. \]
Still by \eqref{estimateint} we get  \[\lim_{\epsilon\to 0}\int_{0}^{T}\int_{\mathbb{T}^n} (V(y, m^1+n^\epsilon)-V(y,m^1))n^\epsilon=0.\] 
For every $K>0$, we write 
\begin{eqnarray} \nonumber  O(\epsilon)&=& \int_{0}^{T}\int_{\mathbb{T}^n} (V(y, m^1+n^\epsilon)-V(y,m^1))n^\epsilon\\ &=& \int_{0}^{T}\int_{\mathbb{T}^n}  1_{\{|n^\epsilon|\leq K\}} (V(y, m^1+n^\epsilon)-V(y,m^1))n^\epsilon
\nonumber \\ && +  \int_{0}^{T}\int_{\mathbb{T}^n}  1_{\{|n^\epsilon|\geq K\}} (V(y, m^1+n^\epsilon)-V(y,m^1))n^\epsilon.   \nonumber\end{eqnarray} 
Taking advantage of  the strict monotonicity  \eqref{incr}, we have that  
$(V(y, m^1+n^\epsilon)-V(y,m^1))n^\epsilon\geq 0$, and also that there exists  $C$ depending on $K$ and $\|m^1\|_\infty$ such that 
\begin{eqnarray}\nonumber  O(\epsilon)  &=&
 \int_{0}^{T}\int_{\mathbb{T}^n}  1_{\{|n^\epsilon|\leq K\}} (V(y, m^1+n^\epsilon)-V(y,m^1))n^\epsilon  \\ 
&\geq &  C \int_{0}^{T}\int_{\mathbb{T}^n}  1_{\{|n^\epsilon|\leq K\}} |n^\epsilon|^2\geq \frac{ C}{T} \left( \int_{0}^{T}\int_{\mathbb{T}^n}  1_{\{|n^\epsilon|\leq K\}} |n^\epsilon|\right)^2.\label{tre} \end{eqnarray}
Still by \eqref{incr}, possibly increasing  $C$, there holds 
\[|V(y, n^\epsilon+m^1)-V(y,   m ^1)|>C \qquad{\rm on\ } 
\{y\in \mathbb{T}^n:\ |n^\epsilon|\geq K\}.\] 
Hence 
\begin{eqnarray} \nonumber
O(\epsilon) &= &\int_{0}^{T}\int_{\mathbb{T}^n}  1_{\{|n^\epsilon|\geq K\}} (V(y, m^1+n^\epsilon)-V(y,m^1))n^\epsilon\\ 
&\geq & C \int_{0}^{T}\int_{\mathbb{T}^n}  1_{\{|n^\epsilon|\geq K\}} |n^\epsilon|.\label{dueprimo}
 \end{eqnarray} In conclusion, by  \eqref{tre} and \eqref{dueprimo}, we get  $n^\epsilon\to 0$ in $L^1 (\mathbb{T}^n\times [0,T])$. 
\end{proof}

\subsection{Proof of Theorem \ref{thm:conv}}\label{proofconv}
We start proving the convergence of $m^\epsilon$ stated in point (ii).
We observe that, by  Lemma \ref{lemmaconv} , we have that
\begin{equation*}
\lim_{\epsilon\to 0}\int_0^T\int_{\mathbb{T}^n}|m^\epsilon(\epsilon y,t) -m^1(y)| dy dt=0.
\end{equation*}
By the change of variables $z=\epsilon y$, we get
\begin{equation*}
\lim_{\epsilon\to 0}\epsilon^{-n}\int_0^T\int_{\epsilon\mathbb{T}^n}|m^\epsilon(z,t) -m^1(z/\epsilon)| dz dt=0.
\end{equation*}
Fix a compact $Q$ in $\mathbb{R}^n$; then there exists a constant $K_Q\in \mathbb{N}$ such that $Q\subset K_Q [0,1]^n$.  So, using periodicity, we have
\begin{eqnarray*}
\int_0^T\int_{Q}|m^\epsilon(z,t) -m^1(z/\epsilon)| dz dt&\leq&
\int_0^T\int_{K_Q[0,1]^n}|m^\epsilon(z,t) -m^1(z/\epsilon)| dz dt=o(1).
\end{eqnarray*}
On the other hand, by weak convergence of a periodic function to its mean (see \cite{luk}), $m^1(x/\epsilon)$ weakly converge to $1$ in $L^1_{\text{loc}}(\mathbb{R}^n)$. Hence, point (ii) of the statement is completely proved.

Now we pass to the proof of i). Integrating on $(0,T)\times \mathbb{T}^n$ the first equation in \eqref{systemp}, we get
\begin{multline}\label{t1}
\int_{\mathbb{T}^n}w^\epsilon (\epsilon x,t)dx +\int_t^T\int_{\mathbb{T}^n} \frac12 |\nabla w^\epsilon (\epsilon x,s)+P|^2dxds  \\-\int_t^T\int_{\mathbb{T}^n} V(x,m^\epsilon(\epsilon x,s))dxds=0.
\end{multline}
Since, by Lemma~\ref{lemmaconv}, there holds
\[
\int_0^T\int_{\mathbb{T}^n} |(\nabla w^\epsilon (\epsilon x,t)+P)-(\nabla u^1(x)+P)|^2 dxdt \leq C\epsilon,
\]
the second term in \eqref{t1} converges to $\int_0^T\int_{\mathbb{T}^n} \frac12|\nabla u^1(x)+P|^2$. On the other hand, we claim that
\[
 V(x,m^\epsilon(\epsilon x,t))\to   V(x,m^1(x))\quad \text{in }L^1([0,T]\times \mathbb{T}^n).
\]
Indeed, following a similar argument as in \cite[Thm2.1]{cllp1}, we get that, calling $L$ the Lipschitz constant of $V$ in $\mathbb{T}^n\times[0, \|m^1\|_\infty+1]$,
\begin{eqnarray*}
&& \int_0^T\int_{\mathbb{T}^n} |V(x,m^\epsilon(\epsilon x,t))-V(x,m^1(x))|dxdt\\
&\leq& \int_0^T\int_{\{\ m^\epsilon(\epsilon x,t)>\|m^1\|_\infty+1\}}\!\! |V(x,m^\epsilon(\epsilon x,t))-V(x,m^1(x))|dxdt
+ L \int_0^T\int_{\mathbb{T}^n} |m^\epsilon(\epsilon x,t)-m^1(x)|dxdt \\
&\leq& \int_0^T\int_{\{\ m^\epsilon(\epsilon x,t)>m^1(x)+1\}}\!\! (V(x,m^\epsilon(\epsilon x,t))-V(x,m^1(x)))(m^\epsilon(\epsilon x,t)-m^1(x))dxdt\\ &&+ L \int_0^T\int_{\mathbb{T}^n} |m^\epsilon(\epsilon x,t)-m^1(x)|dxdt 
\\
&\leq& C_P \epsilon +L \int_0^T\int_{\mathbb{T}^n} |m^\epsilon(\epsilon x,t)-m^1(x)|dxdt=C_P \epsilon +L \int_0^T\int_{\mathbb{T}^n} |n^\epsilon(x,t)|dxdt
\end{eqnarray*}   where the last inequality is due to Lemma~\ref{lemmaconv}.  Again by   Lemma~\ref{lemmaconv},  we get that the right hand side converges to $0$ as $\epsilon\to 0$. Hence, our claim is completely proved.

So, by the definition of $w^\epsilon$ and of $\bar H (P,1)$,  we get from \eqref{t1} that, for every $t\in [0,T]$ there holds
\begin{eqnarray*} 
\lim_{\epsilon\to 0} \int_{\mathbb{T}^n}( u^\epsilon (\epsilon x,t)-\epsilon P\cdot x) dx &= & \int_t^T\int_{\mathbb{T}^n}\left[ -\frac12|\nabla u^1(x)+P|^2+V(x,m^1(x))\right]dxdt
\\&=& (t-T) \bar H(P,1).
\end{eqnarray*}

This implies that the function $\bar v^\epsilon(t):= \int_{\mathbb{T}^n} v^\epsilon(x,t)dx$ satisfies
\[\lim_{\epsilon\to 0}\epsilon\bar v^\epsilon(t)= \lim_{\epsilon\to 0} \int_{\mathbb{T}^n}[ u^\epsilon (\epsilon x,t)-\epsilon P\cdot x -(t-T)\bar H(P,1)-\epsilon u^1(x)]dx=0.  \]  Note that since all the previous estimates are independent of $t$, the convergence is also uniform for $t\in[0,T]$.
By  Poincar\'e inequality, recalling Lemma \ref{lemmaconv}, we get
 \[  \int_0^T \int_{\mathbb{T}^n} |v^\epsilon(x,t)-\bar v^\epsilon|^2 \leq C \int_0^T \int_{\mathbb{T}^n} |\nabla v^\epsilon|^2\leq C_P\epsilon .\] So, we get that
$ v^\epsilon(x,t)-\bar v^\epsilon(t)\to 0$ in $L^2(\mathbb{T}^n\times[0,T])$.
In particular the last two relations imply that $\epsilon v^\epsilon(x,t) \to 0$ in $L^2(\mathbb{T}^n\times[0,T])$, namely \[
\lim_{\epsilon\to 0}\int_0^T\int_{\mathbb{T}^n}| u^\epsilon (\epsilon x,t)-\epsilon P\cdot x-(t-T) \bar H(P,1)-\epsilon u^1(x)|^2dx =0.\] Since $\epsilon u^1\to 0$ in $L^2$ we get that  \[
\lim_{\epsilon\to 0}\int_0^T\int_{\mathbb{T}^n}| u^\epsilon (\epsilon x,t)-\epsilon P\cdot x-(t-T) \bar H(P,1)|^2dx =0.\]
Performing the change of variables $z=\epsilon x$ and choosing the constant $K_Q$ as before we get that
\[\lim_{\epsilon\to 0} K_Q^n \int_0^T\int_{Q}| u^\epsilon (z,t)- P\cdot z-(t-T) \bar H(P,1)|^2dz =0 \] which implies the desired result.

\subsection{Final remark on the convergence in the general case}\label{sect:general}

Here we briefly sketch some arguments for extending the convergence result beyond affine data, but still under restrictive assumptions; this issue will be addressed in a future work. 
Assume that 
\begin{itemize} \item   \eqref{systemeffintro} has a classical solution (this could  hold in a small time interval for regular data),
\item the solution to  \eqref{ucor} has a regular dependence  on  $P$ and $\alpha$,
\item the data are well-prepared: they agree with those of the asymptotic expansions up to a sufficiently high order.\end{itemize}
Let $(u(\cdot; x,t),m(\cdot; x,t))$ be the  solution to \eqref{ucor} with
$(P,\alpha)=(\nabla u^0(x,t), m^0(x,t))$ and let  $ m^2(\cdot; x,t)$  be the solution to \eqref{m2}. 
We put $\epsilon=\frac{1}{k}$, with $k\in \mathbb{N}$.
We define the error \[\begin{cases} v^\epsilon(x,t)=u^\epsilon(x,t)-u^0(x,t)-\epsilon u \left(\frac{x}{\epsilon};  x,t\right)=u^\epsilon(x,t)-\bar {u}^\epsilon(x,t),\\
n^\epsilon(x,t)=m^\epsilon(x,t)-m^0(x,t)\left(m \left(\frac{x}{\epsilon}; x,t\right)+\epsilon m^2\left(\frac{x}{\epsilon}; x,t\right)\right)=m^\epsilon(x,t)-\bar {m}^\epsilon(x,t).\end{cases}\]
The assumption that the data are  well-prepared reads: $ v^\epsilon(x,T)=0$ and $n^\epsilon(x,0)=0$.  This may be relaxed slightly using the alternative expansion in \ref{timedep}. 

For $(u^0,m^0)$ and $(u,m)$, $m^2$ regular, as in
Section \ref{formalexp} we get 
$v^\epsilon, n^\epsilon$ satisfy
\[\begin{cases} -v^\epsilon_t-\epsilon\Delta v^\epsilon+ \frac{1}{2}|\nabla v^\epsilon|^2+\nabla \bar {u}^\epsilon\cdot \nabla v^\epsilon- V(\frac{x}{\epsilon}, n^\epsilon+\bar  m ^\epsilon)+V(\frac{x}{\epsilon}, \bar  m ^\epsilon)= R_1^\epsilon \\
n^\epsilon_t{-\epsilon \Delta n^\epsilon}{-}{\rm div}\left(n^\epsilon(\nabla \bar {u}^\epsilon)\right){-}{\rm div}\left(m^\epsilon(\nabla v^\epsilon)\right)= R_2^\epsilon\end{cases}
\]where  $R_i^\epsilon\to 0$ uniformly as $\epsilon\to 0,$ $i\in \{1,2\}$. 
We multiply  the first equation by $n^\epsilon$ and the second by $v^\epsilon$, integrate on $\mathbb{T}^n\times[0, T]$ and subtract one equation from the other. 
Taking into account that the data are well prepared and 
$\frac{1}{2}n^\epsilon-m^\epsilon=-\frac{1}{2}(m^\epsilon+\bar  m ^\epsilon)<0$,  we get 
\begin{equation}\label{finalest}
\int_0^T\int_{\mathbb{T}^n}\frac{1}{2}(m^\epsilon+\bar  m ^\epsilon)|\nabla v^\epsilon|^2+\left[V\left(\frac{x}{\epsilon}, n^\epsilon+\bar  m ^\epsilon\right)-V\left(\frac{x}{\epsilon}, \bar  m ^\epsilon\right)\right] n^\epsilon dxdt=O(\epsilon).
\end{equation}

Since $\bar  m^\epsilon$ is bounded away from zero and uniformly bounded by our assumptions,  and $m^\epsilon\geq 0$, we can argue as in Lemma \ref{lemmaconv}. Then 
we get that $|\nabla v^\epsilon|\to 0$ in $L^2(\mathbb{T}^n\times [0, T])$. This (together with the  Poincar\'e inequality)  implies that  $v^\epsilon\to  0$ in  $L^2(\mathbb{T}^n\times [0, T])$ and 
moreover  $n^\epsilon\to 0$ in $L^1(\mathbb{T}^n\times [0, T])$. In conclusion  
 $u^\epsilon\to u^0$ in $L^2(\mathbb{T}^n\times [0, T])$ and
%
$m^\epsilon\to m^0$ weakly in $L^1(\mathbb{T}^n\times [0, T])$. 

Let us stress that even in this case,  the oscillations at highest order of $\bar {m}^\epsilon$ prevent to expect any stronger convergence of $m^\epsilon$ to $m^0.$

\end{document}